\documentclass[a4paper, 12pt]{amsart}
\usepackage{amssymb, amsthm, amsmath}
\usepackage{enumerate}
\usepackage{stmaryrd}
\usepackage[all]{xy}
\usepackage{tikz-cd}
\usetikzlibrary{decorations.text}
\usepackage{microtype}

\usepackage{hyperref}
\usepackage{cleveref}

\usepackage{xcolor}

\makeatletter
  \AtBeginDocument{%
    \let\plain@equationautorefname\equationautorefname
    \def\equationautorefname{\plain@equationautorefname\@autoref@insert@tagform}%
    \def\@autoref@insert@tagform~#1\null{~(#1\null)}%
  }
\makeatother

\makeatletter
\newcommand{\getcitenumber}[1]{
	\@ifundefined{b@#1}
	{\hbox{\reset@font\bfseries ?}}
	{\csname b@#1\endcsname}}
\makeatother

\numberwithin{equation}{section}
\newtheorem{thm}[equation]{Theorem}
\newtheorem{lem}[equation]{Lemma}
\newtheorem{prop}[equation]{Proposition}
\newtheorem{cor}[equation]{Corollary}
\newtheorem{rem}[equation]{Remark}

\newtheorem{define}[equation]{Definition}

\newtheorem{introthm}{{}Theorem}
\newenvironment{myintrothm}[2][]
  {\renewcommand\theintrothm{#2}\begin{introthm}[#1]}
  {\end{introthm}}

\newcommand{\FP}{\mathrm{FP}}
\newcommand{\PFP}{\mathrm{PFP}}
\newcommand{\Irr}{\mathrm{Irr}}
\newcommand{\Hom}{\operatorname{Hom}}
\newcommand{\End}{\operatorname{End}}
\newcommand{\Epi}{\operatorname{Epi}}

\newcommand{\Ind}{\operatorname{Ind}}
\newcommand{\Res}{\operatorname{Res}}
\newcommand{\ext}{\operatorname{ext}}
\newcommand{\Aut}{\operatorname{Aut}}
\newcommand{\Sym}{\operatorname{Sym}}
\newcommand{\bra}{\llbracket}
\newcommand{\ket}{\rrbracket}

\newcommand{\rk}{\mathrm{rk}}

\newcommand{\comm}[1]{}

\newcommand{\F}{\mathbb{F}}
\newcommand{\bbF}{\mathbb{F}}
\newcommand{\Z}{\mathbb{Z}}
\DeclareMathOperator{\SL}{SL}
\DeclareMathOperator{\PSL}{PSL}
\DeclareMathOperator{\md}{md}

\begin{document}

\title[Counting irreducible modules]{Counting irreducible modules for profinite groups}
\author[G. Corob Cook]{Ged Corob Cook}
\email{gcorobcook@gmail.com}
\author[S. Kionke]{Steffen Kionke}
\email{steffen.kionke@fernuni-hagen.de}
\author[M. Vannacci]{Matteo Vannacci}
\email{matteo.vannacci@ehu.eus}
\date{\today}

\thanks{S. K. acknowledges support of the Deutsche Forschungsgemeinschaft (DFG, German Research Foundation) -
441848266}


\begin{abstract}
This article is concerned with the representation growth of profinite groups over finite fields. We investigate the structure of groups with uniformly bounded exponential representation growth (UBERG). Using crown-based powers we obtain some necessary and some sufficient conditions for groups to have UBERG. As an application we prove that the class of UBERG groups is closed under split extensions but fails to be closed under extensions in general. On the other hand, we show that the closely related probabilistic finiteness property $\mathrm{PFP}_1$ is closed under extensions. In addition, we prove that profinite groups of type $\mathrm{FP}_1$ with UBERG are always finitely generated and we characterise UBERG in the class of pro-nilpotent groups.

Using infinite products of finite groups, we construct several examples of profinite groups with unexpected properties: (1) an UBERG group which cannot be finitely generated, (2) a group of type $\mathrm{PFP}_\infty$ which is not UBERG and not finitely generated and (3) a group of type $\mathrm{PFP}_\infty$ with superexponential subgroup growth. 
\end{abstract}

\maketitle

\setcounter{tocdepth}{1}
\tableofcontents

\newpage

\section*{Introduction}
In recent years, there has been a growing interest in understanding the asymptotic behaviour of representations of infinite groups and families of finite groups. In particular, a lot of effort has been expended in studying the \emph{representation growth} of \emph{rigid} groups (see \cite{BLMS,LM} for instance), i.e.\ with only finitely many complex representations in each dimension. It turns out that asymptotic representation growth of a group $G$ carries a lot of information about the structure of $G$. Moreover, asymptotic representation-theoretic information about families of finite groups leads to striking results on many fronts: examples include the $(2,3)$-generation of finite simple groups (see \cite{LieS}), and the fact that the representation growth of arithmetic groups is rational (\cite{Av,AKOV}).

Nevertheless, all the above-mentioned results study the representation theory of groups in characteristic zero. A natural question is whether there is a reasonable parallel theory for representations over finite fields. Of course, given a finitely generated group $G$ and a finite field $F$, it is clear that $G$ has only finitely many representations over $F$ of a given degree; hence, the rigidness condition is automatic in this case. Write $r(G,F,n)$ for the number of irreducible representation of $G$ over $F$ of dimension $n$. Note that, if $G$ is $d$-generated, $r(G,F,n)\le \vert F \vert^{dn^2}$, so any reasonable restriction should improve on this bound. Moreover, since we are looking only at homomorphisms into finite groups, it is sufficient to restrict our attention to representations of \emph{profinite groups}, by passing to profinite completions.

We say that a profinite group $G$ has \emph{UBERG} if there exists a constant $c>0$ such that $r(G,F,n)\le \vert F \vert^{cn}$ for every finite field $F$. UBERG stands for `uniformly bounded exponential representation growth' (over finite fields) and, maybe surprisingly, it shows up naturally in the study of probabilistic generation properties of profinite groups. In fact, a finitely presented profinite group is \emph{positively finitely related} (PFR) exactly if it has UBERG. Moreover, a profinite group has UBERG if and only if the completed group algebra $\hat{\mathbb{Z}}\llbracket G\rrbracket$ is \emph{positively finitely generated} (PFG, see \cite{KV}). Nevertheless, the structural properties of UBERG groups are hardly understood and it is unknown this class is closed under extensions. 

The goal of this article is twofold. First, we will prove several fundamental results on UBERG groups that, we hope, will provide a foundation for further study on the modular representation theory of profinite groups. This can be seen as the parallel to the study of ``rigidity'' carried out in \cite{LM}. Secondly, we will build a ``twisted Clifford theory'' for crossed representations (see Theorem~\ref{modulestructure}) that will be one of our main tools to study extensions of UBERG groups.

Additionally, we concentrate our study on the representation theory of cartesian products of finite (simple) groups and we provide several examples with different asymptotic behaviour in the number of modular representations (see Section~\ref{sec:examples}). In particular, we can answer various questions that were left open in \cite{CCV}; for instance, we settle Open Questions~6.3, 6.4 and 6.6 of \cite{CCV}. Moreover, we give an example to show that extensions of PFR groups do not need to be PFR, answering a question raised in \cite[p.3]{KV}. 

Finally, one should note that, even though the characteristic-zero representation theory of finite groups is relatively well-understood, the modular representation theory of finite groups is still in many respects a mystery. Therefore, an asymptotic approach could be highly desirable and we hope that UBERG groups can provide a new framework to study modular representations asymptotically.

For convenience, we provide a diagram showing the relationships between the various conditions studied in the paper. 

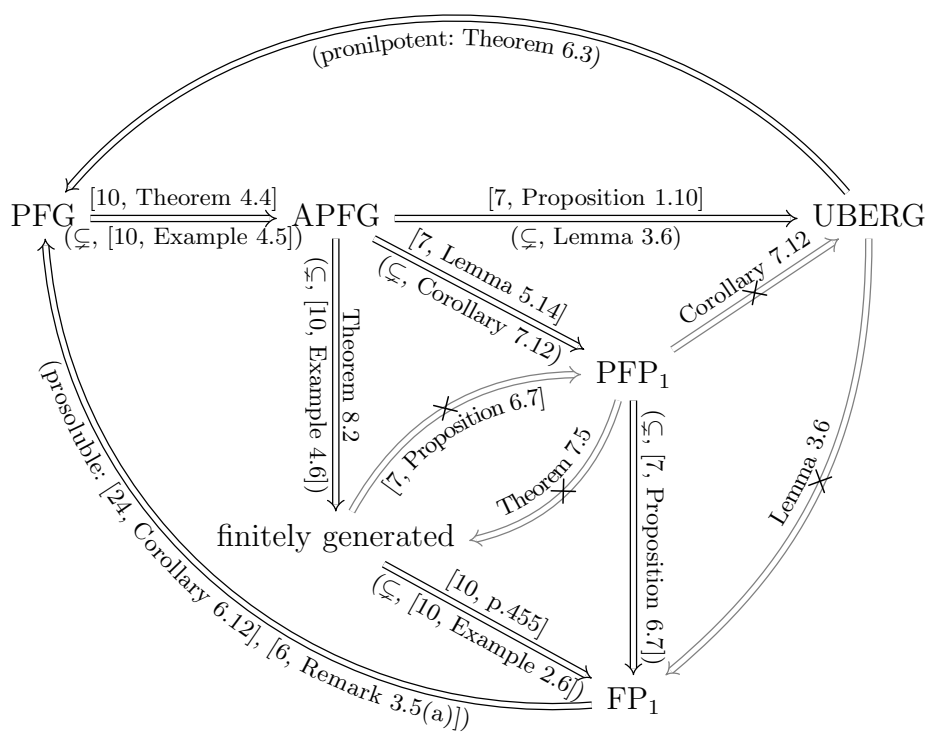
\[\begin{NoHyper}\begin{tikzcd}[math mode=false,labels = {font = \scriptsize},sep=large,row sep=huge]
	PFG \arrow[r,Rightarrow,"{\cite[Theorem 4.4]{Damian}}","{($\subsetneq$, \cite[Example 4.5]{Damian})}"']
	& APFG \arrow[rr,Rightarrow,"{\cite[Proposition 1.10]{CCV}}" sloped,"{($\subsetneq$, \Cref{lem:UBERG_inf_gen})}"' sloped]
	\arrow[dr,Rightarrow,"{\cite[Lemma 5.14]{CCV}}" sloped,"{($\subsetneq$, \Cref{superexponential})}"' sloped]
	\arrow[dd,Rightarrow,"{\Cref{thm:UBERG+FP_1=FG}}" sloped,"{($\subsetneq$, \cite[Example 4.6]{Damian})}"' sloped]
	&& UBERG \arrow[lll,Rightarrow,bend right=50,postaction={decoration={text along path, text={|\scriptsize|(pronilpotent: Theorem \getrefnumber{thm:pronilpotent})},raise=-10pt,reverse path,text align={align=center}},decorate}]
	\arrow[dddl,Rightarrow,bend left=25, gray, "{\large$\textcolor{black}{\times}$}" marking, postaction={decoration={text along path, text={|\scriptsize|Lemma \getrefnumber{lem:UBERG_inf_gen}},raise=5pt,reverse path,text align={align=center}},decorate}] \\
	&& $\PFP_1$ \arrow[dl,Rightarrow,bend left, gray,"{\large$\textcolor{black}{\times}$}" marking,postaction={decoration={text along path, text={|\scriptsize|Theorem \getrefnumber{thm:product-of-sl2s}},raise=5pt,reverse path,text align={align=center}},decorate}]
	\arrow[dd,Rightarrow,"{($\subsetneq$, \cite[Proposition 6.7]{CCV})}" sloped]
	\arrow[ur,Rightarrow, gray, "{\large$\textcolor{black}{\times}$}" marking, "\textcolor{black}{\Cref{superexponential}}" sloped] \\
	& finitely generated \arrow[ur,Rightarrow, gray, bend left,"{\large$\textcolor{black}{\times}$}" marking,postaction={decoration={text along path, text={|\scriptsize|[{\getcitenumber{CCV}}, Proposition 6.7]},raise=-10pt,text align={align=center}},decorate}]
	\arrow[dr,Rightarrow,"{\cite[p.455]{Damian}}" sloped, "{($\subsetneq$, \cite[Example 2.6]{Damian})}"' sloped] \\
	&& $\FP_1$ \arrow[uuull,Rightarrow,bend left=45,postaction={decoration={text along path, text={|\scriptsize|(prosoluble: [{\getcitenumber{KV}}, Corollary 6.12], [{\getcitenumber{Ged}}, Remark 3.5(a)])},raise=-10pt,reverse path,text align={align=center}},decorate}]
\end{tikzcd}\end{NoHyper}\]

For implications in this diagram marked $(\subsetneq,-)$, this reference provides a counterexample showing the reverse implication fails; for those marked `pronilpotent' (respectively, `prosoluble'), the implication holds for pronilpotent (respectively, prosoluble) profinite groups, though not in general. (Non-)implications which follow from those marked in the diagram may be left unmarked.

\subsection*{Main results}

It is not surprising that there is a direct connection between the growth of (linear or projective) representations of a profinite group and the theory of \emph{crowns} associated to composition factors (cf.\ Section~\ref{sec:crowns}). Our first main result makes this correspondence explicit. In \cite{JP}, the invariant $l(G)$ (the minimal degree of a faithful transitive representation of a group $G$) was used to characterise the \emph{positive finite generation} (PFG) property for profinite groups. In fact, in \cite{JP}, it is shown that a profinite group $G$ is PFG if and only if there is some constant $c$ such that, for any monolithic group $L$ with non-abelian minimal normal subgroup $N$ and any $k$ such that the crown-based power $L_k$ appears has a quotient of $G$, $k\le l(N)^c$ (again, see Section~\ref{sec:crowns} for the basic definitions). Here we look at two related invariants which will provide one necessary and one sufficient condition for UBERG, respectively. See Section~\ref{sec:cond_UBERG} for the definition of $l^{proj}(K)$ and $l^{lin}(K)$ for $K$ a non-abelian characteristically simple group.

\begin{myintrothm}{A}
	\label{introthm:charUBERG}Let $G$ be a profinite group.
	
 	\begin{enumerate}[(i)]
	\item \label{introthm:charUBERG_1} Suppose $G$ is finitely generated. Suppose there is some $b$ such that, for all finite monolithic groups $L$ with non-abelian minimal normal subgroup $K$, if the crown-based power $L_k$ is a quotient of $G$ then $k\le l^{proj}(K)^b$. Then $G$ has UBERG.
	\item \label{introthm:charUBERG_2}Suppose, for all $b$, there is some monolithic group $L_b$ with non-abelian minimal normal subgroup $K_b$ such that some crown-based power $(L_b)_k$ of $L_b$ is isomorphic to a quotient of $G$ and $k> l^{lin}(K_b)^b$. Then $G$ does not have UBERG.
	\end{enumerate}
\end{myintrothm}

We remark that there are UBERG groups that do not satisfy the condition \eqref{introthm:charUBERG_1} of \cref{introthm:charUBERG}, for instance the group $H$ from Theorem~\ref{introthm:example-procyclic-by-uberg} below.

Next, we address the question of whether the property of having UBERG is closed under extensions. Note that the corresponding question is easily seen to hold for PFG groups (see \cite[Proposition 7]{Mann}). However, UBERG-by-UBERG groups are not UBERG in general. Here we exhibit a procyclic-by-UBERG non-UBERG profinite group.

\begin{myintrothm}{B}
\label{introthm:example-procyclic-by-uberg}
Let $(n_i)_{i\in \mathbb{N}}$ be an increasing sequence of pairwise coprime integers $\geq 12$ and let $q_i = p_i^{k_i}$ be a sequence of prime powers for pairwise distinct primes $p_i \geq 5$ such that $\gcd(n_i,q_i-1)>1$. Let $m_i = q_i^{\lfloor n_i^{3/2} \rfloor}$. 
Consider the profinite group
\[
	G = \prod_{i \in \mathbb{N}} \SL_{n_i}(\bbF_{q_i})^{m_i}.
\]
Then $G$ is $2$-generated, finitely presented and it does not have UBERG. Moreover, we can choose a procyclic central subgroup $Z\le G$ such that the quotient group $H=G/Z$ is $2$-generated, finitely presented and it has UBERG.
\end{myintrothm}
The subgroup $Z$ in the previous theorem is defined explicitly in Section~\ref{sec:UBERGbyPUBERG_count}.

Even though it is not true that UBERG-by-UBERG groups are UBERG in full generality, we develop a Clifford theory from crossed representations which might be of independent interest (see Theorem~\ref{modulestructure}). In fact, as a first application, we will use Theorem~\ref{modulestructure} to show that split extensions of UBERG groups are UBERG.

\begin{myintrothm}{C}
\label{introthm:UBERGbyUBERG}
	Suppose $G$ is a profinite group, with $K \unlhd G$.
	\begin{enumerate}[(i)]
		\item If $K$ and $G/K$ have UBERG, and the extension of $K$ by $G/K$ is split, then $G$ has UBERG.
		\item If $K$ has UBERG and $G/K$ is PFG, then $G$ has UBERG.
	\end{enumerate}
\end{myintrothm}

In \cite{CCV}, certain probabilistic versions of the cohomological finiteness properties type $\FP_n$ were introduced, see Section~\ref{sec:PFG_PFR_PFP}. In this article we continue the study of the first of these (i.e.\ type $\PFP_1$) and we provide a semi-structural criterion in the spirit of Theorem~\ref{introthm:charUBERG} for type $\PFP_1$ (see Theorem~\ref{thm:PFP1}). It turns out that $G$ having type $\PFP_1$ is related to the growth of the number of $G$-isomorphism classes of non-Frattini abelian chief factors of $G$ (see Definition~\ref{defn:A}) and the growth of the size of the smallest faithful irreducible representation $M$ of monolithic quotients $L$ of $G$ with $H^1(L,M)\neq 0$ (see Definition~\ref{defn:B_C}). The aforementioned characterisation allows us to show that extensions of type $\PFP_1$ groups have type $\PFP_1$ (see Theorem~\ref{thm:PFP1_by_PFP1}). 

\begin{myintrothm}{D}\label{introthm:PFP1}
 Suppose $G$ is a profinite group, $K \unlhd G$. If $K$ and $G/K$ have type $\PFP_1$, then $G$ has type $\PFP_1$.
\end{myintrothm}

As often happens in the world of profinite groups, it is hard to distinguish different group-theoretic properties and the construction of explicit examples is very desirable. 

Here we completely characterise pronilpotent UBERG groups (see Theorem~\ref{thm:pronilpotent}); these are exactly the finitely generated pronilpotent groups.

\begin{myintrothm}{E}\label{introthm:pronilp}
	Let $P$ be pronilpotent group. The following are equivalent:
	\begin{enumerate}[(i)]
		\item $P$ is finitely generated,
		\item $P$ has UBERG,
		\item $P$ is of type $\PFP_1$,
		\item $P$ is of type $\FP_1$.
	\end{enumerate}
\end{myintrothm}

Next, we concentrate on the class of cartesian products of finite groups and we produce several examples of cartesian product with various subsets of the properties described above, in particular, finite generation, UBERG and type $\PFP_1$:

\begin{enumerate}[(i)]
	\item We will exhibit a type $\PFP_1$ group which is not finitely generated and does not have UBERG (see Theorem~\ref{thm:product-of-sl2s}); using the theory of \emph{universal Frattini covers}, we can even construct a projective profinite group of type $\PFP_\infty$ which is not finitely generated. This is in sharp contrast to the case of abstract groups, for which type $\FP_1$ is equivalent to finite generation.
	\item Similarly, we construct a $2$-generated profinite group of type $\PFP_1$, but without UBERG (Corollary~\ref{superexponential}). Indeed, this example has superexponential subgroup growth, which is impossible for groups with UBERG by \cite[Corollary 5.5]{KV}.
	\item Finally, we have a non-finitely generated metabelian group with UBERG (Lemma~\ref{lem:UBERG_inf_gen}), but which does not have type $\FP_1$, and hence does not have type $\PFP_1$.
 \end{enumerate}
However, we show that a $\PFP_1$ group with UBERG must be finitely generated (see Theorem~\ref{thm:UBERG+FP_1=FG}). In fact, Theorem \ref{thm:UBERG+FP_1=FG} shows more: if we assume that our group has type $\FP_1$, then UBERG implies finite generation.

\begin{myintrothm}{F}\label{introthm:UBERG_FP_1_FG}
 	Suppose $G$ is a profinite group with UBERG and type $\FP_1$. Then $G$ is finitely generated.
\end{myintrothm}

Since it was shown in \cite[Proposition 1.10]{CCV} that the UBERG and type $\FP_1$ conditions together are equivalent to the APFG condition (see Section~\ref{sec:PFG_PFR_PFP}), we can express this result by saying that if the augmentation ideal $\ker(\hat{\Z}\bra G \ket \to \hat{\Z})$ is PFG, then $G$ is finitely generated.

This should be very surprising: it is almost an axiom of homological algebra that the choice of which projective resolution we use should not matter -- and certainly this difference is not detectable by any (co)homology groups -- but nonetheless, generation properties of the kernel of the projective cover of $\hat{\Z}$ in the category of $\hat{\Z}\bra G \ket$-modules cannot tell us whether $G$ is finitely generated, and generation properties of the kernel of the augmentation map $\hat{\Z}\bra G \ket \to \hat{\Z}$ do.

\subsection*{Organisation of the article}
In Section~\ref{sec:prelim} we start by giving the basic definitions and fixing the notation that we need in the rest of the article. In Section~\ref{sec:cond_UBERG}, we prove some bounds on the sizes of (linear and projective) representations of monolithic groups and use these to prove Theorem~\ref{introthm:charUBERG}. In Section~\ref{sec:inf_prod_fin_grps}, we give conditions for an infinite product of finite groups to have UBERG, to prove Theorem~\ref{introthm:example-procyclic-by-uberg} and construct an infinitely generated group with UBERG. Section~\ref{sec:UBERG-by-UBERG} is devoted to developing our Clifford theory for twisted modules and contains the proof of Theorem~\ref{introthm:UBERGbyUBERG}, as well as analogous results on type $\PFP_n$. In Section~\ref{sec:PFP1}, we characterise groups of type $\PFP_1$ in terms of crown-based powers appearing as quotients of the group, and use this to prove Theorem~\ref{introthm:PFP1}. The proof of Theorem~\ref{introthm:pronilp} can be found in Section~\ref{sec:pronilp}. Section~\ref{sec:examples} is our second source of interesting examples, especially of the groups promised above which have type $\PFP_1$ but not UBERG. Finally, in Section~\ref{sec:UBERG_FP1}, we prove Theorem~\ref{introthm:UBERG_FP_1_FG}.


\section{Preliminaries, terminology and notation}\label{sec:prelim}

\subsection{Notation}\label{sec:notation}
 As it is customary when working with profinite groups, we will assume that subgroups are closed, maps are continuous. Furthermore, generation will be intended in the topological sense. The same will be assumed for profinite modules.

For $F$ a field, $F^\times$ is the group of non-zero elements under multiplication.

 Let $G$ be a finite group. The \emph{socle} $\mathrm{soc}(G)$ is the subgroup generated by all minimal normal groups in $G$.  We denote by $E(G)$ the subgroup generated by all quasisimple subnormal subgroups of $G$. This is sometimes called the \emph{layer} of $G$ and it forms part of the generalised Fitting subgroup of $G$.

\subsection{Projective and crossed representations, and cocycles}\label{sec:projective-crossed-etc}

We will use the language of crossed representations and crossed projective representations following \cite{Karpilovsky}.
Let $E$ be a field and let $G$ be a profinite group. A \emph{representation} of $G$ over $E$ of degree $n$ is a homomorphism $\rho \colon G \to GL_n(E)$.

A semilinear transformation of an $E$-vector space $V$ is an additive homomorphism $f: V \to V$ such that there exists an automorphism $\phi$ of $E$ with $f(\lambda v) = \phi(\lambda)f(v)$ for all $\lambda \in E$ and all $v \in V$. The group of bijective semilinear transformations of $V$ is written $\Gamma L_E(V)$. 
A \emph{crossed representation} of $G$ on $V$ is a homomorphism $\rho \colon G \to \Gamma L_E(V)$. Via the canonical homomorphism $\Gamma L_E(V) \to \Aut(E)$, every crossed representation gives rise to an action $\gamma$ of $G$ on $E$ by field automorphisms. We may say $\rho$ is a $\gamma$-crossed representation of $G$ over $E$.
There is a $1$-to-$1$ correspondence (described in \cite{Karpilovsky}) between $\gamma$-crossed representations of $G$ over $E$ and modules for the ring $E^\gamma \bra G \ket$, which we define as the free profinite $E$-module with basis $\{\bar{g}: g \in G\}$, and multiplication defined distributively by $\bar{g}\bar{h}=\overline{gh}$, $\bar{g}\lambda = \gamma_g(\lambda)\bar{g}$. We will identify $\gamma$-crossed representations with modules for this twisted group ring via this correspondence.

Let $E$ be a finite field and let $V$ be a finite-dimensional $E$-vector space. A \emph{projective crossed representation} $\rho$ of $G$ on $V$ is a map $G \to \Gamma L_E(V)$ such that there is $\alpha \in Z^2_\gamma(G,E^\times)$, where $Z^2_\gamma(G,E^\times)$ is the group of $2$-cocycles for $G$ with respect to some action $\gamma$ of $G$ on $E$, such that $\rho(g)\rho(h) = \alpha(g,h)\rho(gh)$ for all $g,h \in G$, and $\rho(1) = 1$; see \cite{Karpilovsky} for the definition of $2$-cocycles. When it is clear, we may suppress the subscript $\gamma$ from the notation. The $G$-action induced on $E$ by $\rho$, as described above, is the same as $\gamma$. We may say $\rho$ is an $\alpha$-representation of $G$ over $E$. For $F$ a subfield of $E$, we will also say that $\rho$ is $F$-linear if $\gamma(G) \leq \operatorname{Aut}_F(E)$. As above, there is a $1$-to-$1$ correspondence between $\alpha$-representations of $G$ over $E$ and modules for the ring $E^\alpha \bra G \ket$, which we define as the free profinite $E$-module with basis $\{\bar{g}: g \in G\}$, and multiplication defined distributively by $\bar{g}\bar{h}=\alpha(g,h)\overline{gh}$, $\bar{g}\lambda = \gamma_g(\lambda)\bar{g}$ (see \cite[Theorem 14.3]{Karpilovsky}). 
We will identify $\alpha$-representations with modules for this crossed product via this correspondence. 

Finally, let $V=E^n$. A \emph{projective representation} of $G$ of degree $n$ over $E$ is a projective crossed representation $\rho\colon G \to \Gamma L_E(V)$ with trivial $G$-action $\gamma$ on $E$. Hence, such a projective representation induces a homomorphism $G \to \mathrm{PGL}_n(E)$, which we also write as $\rho$. If $E$ is a finite field, we define the \emph{size} of $\rho$ to be $|E|^n$. To every homomorphism $\rho: G \to \mathrm{PGL}_n(E)$ there is attached a well-defined cohomology class $\alpha \in H^2(G,E^\times)$ with respect to the trivial action of $G$ on $E^\times$. This $\rho$ lifts to a representation exactly if the associated cohomology class $\alpha$ is trivial. When we speak of a \emph{faithful} projective representation, we mean one such that $\ker(\rho\colon G \to PGL_n(E)) \subseteq Z(G)$; a \emph{non-trivial} projective representation will mean one such that the induced map $\rho\colon G \to PGL_n(E)$ is non-trivial.

We say two projective representations $\rho_1,\rho_2$ of $G$ of degree $n$ over $E$ are \emph{projectively equivalent} if there is some $x \in \mathrm{PGL}_n(E)$ such that the induced maps $\rho_1,\rho_2: G \to \mathrm{PGL}_n(E)$ satisfy $x^{-1}\rho_1(g)x=\rho_2(g)$ for all $g \in G$.

\begin{lem}\label{lem:cocycles}
Let $G$ be a profinite group with an action $\gamma$ on $E$.
\begin{enumerate}[(i)]
\item Let $\rho$ be a projective crossed $\gamma$-representation of $G$ on $V$ with cocycle $\alpha$. Then the dual representation on $V^* = \Hom_E(V,E)$ is a projective crossed $\gamma$-representation of $G$ with cocycle cohomologous to $\alpha^{-1}$.
\item Let $\rho_1, \rho_2$ be two projective crossed $\gamma$-representations of $G$ on $E$-vector spaces $V_1, V_2$  with  cocycles $\alpha_1, \alpha_2$. Then $\rho_1\otimes_E \rho_2$ is a projective crossed representation on $V_1 \otimes_E V_2$ with cocycle $\alpha_1\alpha_2$.
\end{enumerate}
\end{lem}
\begin{proof}
This follows from simple calculations. For instance, the projective crossed representation of $G$ on $V^*$ is defined as 
$(^gf)(v) = \gamma_g(f(g^{-1}v))$ and thus
\begin{align*}
	(^{g}(^{h}f))(v) &= \gamma_{g}(\gamma_h(f(h^{-1}(g^{-1}v)))) = \gamma_{gh}( f(\alpha(h^{-1},g^{-1}) (gh)^{-1}v))\\
	 &= \gamma_{gh}(\alpha(h^{-1},g^{-1})) \;\gamma_{gh}(^{gh}f(v)).
\end{align*}
It follows from the cocycle identity that $(g,h) \mapsto \gamma_{gh}(\alpha(h^{-1},g^{-1}))$ is cohomologous to $\alpha^{-1}$.
\end{proof}
The first assertion of the lemma allows one to transform simple $E^{\alpha}\bra{G}\ket$-modules into simple $E^{\alpha^{-1}}\bra{G}\ket$-modules by taking duals. 
If $\alpha$ represents the trivial class in $H^2(G,E^\times)$, then $E^{\alpha}\bra{G}\ket \cong E^\gamma\bra{G}\ket$.
In particular, $V \otimes_E V^*$ is an $E^\gamma\bra{G}\ket$-module.

\subsection{Quasiequivalent representations}

Recall that two representations $\rho_1,\rho_2$ of a finite group $G$ over a field $F$ are said to be \emph{quasiequivalent} if there exists $\phi\in \mathrm{Aut}(G)$ such that $\rho_1$ and $\rho_2 \circ  \phi $ are equivalent. Note that quasiequivalence is an equivalence relation and, for faithful representations $\rho_1,\rho_2$, $\rho_1(G)$ and $\rho_2(G)$ are conjugate in $\mathrm{GL}_F(V)$ if and only if they are quasiequivalent (see \cite[Lemma 2.10.14]{KL}).

We say in addition that two projective representations $\rho_1,\rho_2$ of $G$ over $F$ are quasiequivalent if there exists $\phi \in \mathrm{Aut}(G)$ such that $\rho_1$ and $\rho_2 \circ  \phi $ are projectively equivalent.

\subsection{PFG, PFR, UBERG, \texorpdfstring{$\PFP_n$}{PFPn}, ...}
\label{sec:PFG_PFR_PFP}
In this section we recall some basic definitions. The reader can find more information in \cite{KV} and \cite{CCV}. 

\subsubsection{PFG, PFR, UBERG}
We say that a profinite group $G$ is \emph{PFG} if there is a positive integer $k$ such that the probability of $k$ Haar-random elements of $G$ generating the whole group is positive. This condition has been studied extensively and here we only mention the Mann-Shalev theorem \cite[Theorem 4]{MS}: a profinite group $G$ is PFG if and only if it has polynomial maximal subgroup growth.

\begin{rem}
	
	Note that there are unfortunate naming conventions fixed in the literature here. `Polynomial' growth in similar contexts always means `at most polynomial' growth (see, for example, \cite{LM}), so we include, in our definition of PFG, groups which have maximal subgroup growth slower than any polynomial. On the other hand, `exponential' growth (which we will encounter below) usually means that the function in question has the growth type of an exponential function: that is, it is bounded above and below by exponentials.
\end{rem}

In the spirit of the Mann-Shalev theorem, two of the present authors study in \cite{KV} a related property called PFR. We list below some of the conditions considered there that we will need; the interested reader may check \cite{KV} for more details.

A profinite group $G$:
\begin{enumerate}[(i)]
\item is \emph{PFR} if it is finitely generated, and for every epimorphism $f:H \to G$ with $H$ finitely generated, the kernel of $f$ is positively finitely normally generated in $H$;
\item has \emph{UBERG}, if there exists a constant $c>0$ such that, for every finite field $F$ and every $n\in \mathbb{N}$, $r(G,F,n) \le \vert F \vert^{cn} $. 
\end{enumerate}

\begin{prop}[\cite{KV}]
\label{KV}
UBERG is equivalent to the group algebra $\hat{\mathbb{Z}}\llbracket G \rrbracket$ being PFG, for all profinite groups $G$. 
PFR and UBERG are equivalent for finitely presented profinite groups.
\end{prop}

Note that the equivalence of UBERG to $\hat{\mathbb{Z}}\llbracket G \rrbracket$ being PFG is only stated in \cite{KV} for finitely generated groups, but the proof for general groups goes through without change.

In \cite{CCV}, it was shown that there are groups with UBERG which are not PFG. In Section~\ref{subsec:inf_gen_UBERG} we will show that there are non-finitely generated groups with UBERG. 

We also recall the following result from \cite{CCV}.
\begin{prop}[{\cite[Proposition 1.3]{CCV}}]
\label{UBERG=cb}
If $G$ has UBERG, then it is countably based.
\end{prop}

\subsubsection{Type $\PFP_n$}

In \cite{CCV}, the notion of a profinite group of type $\PFP_n$ was introduced. We report the definition here for convenience. Define a module to have \emph{type $\PFP_n$} over a profinite ring $R$ if it has a projective resolution $\ldots \to P_n \to \ldots \to P_1\to P_0 \to R \to 0$ with $P_0, \ldots, P_n$ PFG profinite $R$-modules. 

A profinite group $G$ has \emph{type $\PFP_n$ over $R$} if $R$ has type $\PFP_n$ as $R \llbracket G \rrbracket$-module. 
Unless specified otherwise, type $\PFP_n$ will mean over $\hat{\mathbb{Z}}$.

\subsubsection{APFG}
Finally, the notion of APFG group was introduced by Damian in \cite{Damian}. We recall here the definition. A profinite group $G$ is said to be \emph{APFG} if the augmentation ideal $I_{\hat{\mathbb{Z}}}\llbracket G \rrbracket$ of the completed group algebra $\hat{\mathbb{Z}}\llbracket G \rrbracket$ is PFG as a $\hat{\mathbb{Z}}\llbracket G \rrbracket$-module. This is equivalent to $G$ having UBERG and type $\FP_1$ by \cite[Proposition 1.10]{CCV}.

\subsection{Crowns in groups}\label{sec:crowns}

\subsubsection{Frattini subgroup}

The \emph{Frattini subgroup} $\Phi(G)$ of a profinite group $G$ is $$ \Phi(G)= \bigcap_{M\in \mathcal{M}} M$$ where $\mathcal{M}$ is the set of all open maximal subgroup of $G$. Since the Frattini subgroup of a finite group is nilpotent, it follows that the Frattini subgroup of a profinite group is pronilpotent (see \cite[Corollary 2.8.4]{RZ}).

\subsubsection{\texorpdfstring{$G$}{G}-equivalence}

Let $G$ be a group. A \emph{$G$-group} $A$ is a group together with a homomorphism $\theta: G\to \mathrm{Aut}(A)$ and we write $\theta(g)(a) = a^g$ for convenience. Two $G$-groups $A$ and $B$ are said to be \emph{$G$-isomorphic} (in symbols $A\cong_G B$) if there exists an isomorphism $\varphi:A\to B$ such that $\varphi(a^g) = \varphi(a)^g$ for all $g\in G$, $a\in A$. Two $G$-groups $A$ and $B$ are said to be \emph{$G$-equivalent} (in symbols $A\sim_G B$) if there exist two isomorphisms $\varphi:A\to B$ and $\Psi : A\rtimes G \to B\rtimes G$ such that the following diagram is commutative:
  $$\xymatrix{ 1 \ar[r] & A \ar[r] \ar[d]^\varphi &  A\rtimes G  \ar[r] \ar[d]^\Psi & G \ar[d]^{\mathrm{id}} \ar[r] & 1 \\  1 \ar[r] & B  \ar[r]
&  B\rtimes G \ar[r]  & G  \ar[r] & 1}$$

\subsubsection{Crowns}
Let $G$ be a finite group and let $X/Y = S^t$ be a chief factor of $G$. If $S$ is abelian and $S=C_p$, then conjugation gives a $t$-dimensional irreducible representation of $G$ over $\mathbb{F}_p$. If $S$ is non-abelian instead, remembering that $\mathrm{Aut}(S^t) = \mathrm{Aut}(S)\wr \mathrm{Sym}(t)$, conjugation gives a transitive permutation representation of $G$ of degree $t$.

We now recall several definitions that will be used in many proofs throughout the rest of the article. Recall that a finite group $L$ is called \emph{monolithic} if $L$ has a unique minimal normal subgroup $N$. In this case
the socle $\mathrm{soc}(L)$ is the unique minimal normal subgroup. 
 If in addition $N$ is not contained in $\Phi(L)$, then $L$ is called a \emph{monolithic primitive group}. 

\begin{rem}
	Let $L$ be a monolithic primitive group with minimal normal subgroup $N\not\le \Phi(L)$. Then there exists a maximal subgroup $M$ of $L$ which does not contain $N$. It follows that $M$ is core-free in $L$, and hence that $L$ has a faithful primitive permutation action on the (left) cosets of $M$.
\end{rem}
 

We say that a chief factor $X/Y$ of a finite group $G$ is \emph{Frattini} if $\Phi(G/Y)\ge X/Y$. Any non-abelian chief factor is non-Frattini and any non-Frattini chief factor is complemented. 

Now, given a non-Frattini chief factor $A$ of a finite group $G$, we define $L_A=G/C_G(A)$ if $A$ is non-abelian, and $L_A=(G/C_G(A))A$ if $A$ is abelian. Then $L_A$ is a monolithic primitive group, and we say it is the monolithic primitive group \emph{associated} to $A$.

Let $L$ be a monolithic primitive group and let $N$ be its minimal normal subgroup. For a positive integer $k$, let $L^k$ be the $k$-fold direct product of $L$. The \emph{crown-based power of $L$ of size $k$} $L_k$ is the preimage of the diagonal copy of $L/N$ in $(L/N)^k$, under the projection map $L^k \to (L/N)^k$.

For $A$ a non-Frattini chief factor of $G$ as before, let $\mathcal{N}_A$ be the set of normal subgroups $N$ of $G$ such that $G/N\cong L_A$ and $\mathrm{soc}(G/N)\sim_G A$. Then, setting $R_G(A) = \bigcap_{N\in \mathcal{N}_A} N$, we have that $G/R_G(A)$ is isomorphic to the crown-based power $(L_A)_{\delta_G(A)}$, where $\delta_G(A)$ is the number of non-Frattini chief factors of $G$ $G$-equivalent to $A$ (in any chief series).
 
We recall a standard lemma that we will need later.
\begin{lem}[{\cite{KK}}]\label{lem:KK}
	Let $T$ be a monolithic group with non-abelian minimal normal subgroup $N=S^s$ and fix a copy $S_1$ of $S$ in $N$. Then $T$ embeds in $\mathrm{Aut}(S)\wr T/\tilde{K}$, where $\tilde{K} = \mathrm{core}_T(N_T(S_1))$.
\end{lem}

The theory of crowns in profinite groups is developed in \cite{DL}, where the reader can find the relevant details and proofs.

\subsection{The constant \texorpdfstring{$c_4$}{c4}}\label{sec:c4}

In this article we will make a heavy use of the constant $c_4$ that appears in \cite{JP}, therefore we report its definition here for the convenience of the reader. In \cite{JP}, $c_4$ is defined to be $16 +\max\{3,c_3\}$, for another constant $c_3$. In particular, $c_4 \geq 19$.

The existence of $c_3$ is one of the main results of \cite{LS}; in fact they show this result for $c_3$ an explicit constant. We do not know what values $c_3$ can take.

\section{Conditions for UBERG}\label{sec:cond_UBERG}

Recall that a profinite group $G$ is said to have UBERG if the completed group ring $\hat{\mathbb{Z}}\bra G \ket$ is positively finitely generated as a (left) module for itself. Whereas PFG groups are closed under extensions, this has hitherto been unknown for UBERG groups.

In this section we give related conditions, one necessary and one sufficient, for a finitely generated profinite group $G$ to have UBERG.

We remark that the statement of \cite[Proposition 4.1]{JP} can be sharpened as follows:

\begin{lem}
	\label{transitive}
	Let $G$ be a finite $d$-generated group and $T$ a transitive
	group of degree $n$. Then there are at most $16^{dn} |T| r$ epimorphisms from $G$ onto $T$, where $r$ is the maximum of $\rk_K(G)$ over all $K \cong Alt(b)^s$ such that $bs \leq n$.
\end{lem}

This is the same idea as \cite[Remark 4.2]{JP}, but in the form we need for our purposes.

Suppose $K$ is a non-abelian characteristically simple finite group, say $K = S^s$ with $S$ simple. As in \cite{JP}, we write $l(K)$ for the minimal degree of a faithful transitive permutation representation of $K$. We define the \emph{projective length} $l^{proj}(K)$ of $K$ as follows: let $l^{proj}(S)$ be the smallest size of a non-trivial irreducible projective representation of $S$ (over any field); define $l^{proj}(K) = l^{proj}(S)^s$ (cf.\ Proposition \ref{primitiverep}). Note that $l^{proj}(S)$ is also the smallest size of a non-trivial linear representation of the universal central extension of $S$ (see e.g. \cite{Karpilovsky}).

\begin{lem}
	\label{lengths}
	Suppose $K=S^s$, with $S$ a non-abelian simple group. Then $l^{proj}(K) > l(K)$.
\end{lem}
\begin{proof}
	Since $l(K) \leq  l(S)^s$ (here in fact equality holds), it is enough to show $l^{proj}(S) \geq l(S)$. Given a non-trivial projective representation $S \to GL(M)$ of minimal size over some finite field $F$, we have an action of $S$ on the projective points of $M$, i.e. the orbits of $M \setminus \{0\}$ under the action of $F^\times$. This action is non-trivial because $S$ has non-trivial image in $PGL(M)$; since $S$ is simple, each non-trivial $S$-orbit of this action is a faithful transitive permutation representation of $S$ of size at least $l(S)$, and strictly less than $|M| = l^{proj}(S)$.
\end{proof}

For completeness we include following lemma, which was mentioned in \cite[Theorem 11.1]{JP}.
\begin{lem}
	\label{lem:aut-1}
	Suppose $T$ is a monolithic group with non-abelian minimal normal subgroup $K=S^r$, with $S$ a non-abelian simple group. 
	Then $|\Aut_1(T)| \leq r |\Aut(S)|^r$, where $\Aut_1(T)$ denotes the group of automorphisms of $T$ which induce the identity on $T/K$.
\end{lem}
\begin{proof}
	The automorphism group $\Aut(K)$ is isomorphic to the wreath product $\Aut(S)^r \rtimes \Sym(r)$. The homomorphism from $\Aut(T)$ to $\Aut(K)$ is injective \cite[Lemma 2.10]{JP}.  Since $K$ is minimal, $T$ acts transitively on the direct factors  $S_1 \times \dots \times S_r$ of $S^r$. We observe that the image of $f \in \Aut_1(T)$ in $\Sym(r)$ is uniquely determined by $f(S_1)$. Indeed,  suppose $S_j = tS_1t^{-1}$ for $t \in T$, then $f(t) = tk$ for some $k \in K$ and hence
	$f(tS_1t^{-1}) = f(t)f(S_1)f(t)^{-1} = tkf(S_1)k^{-1}t^{-1} = tf(S_1)t^{-1}$.
Thus the image of $\Aut_1(T)$ in $\Sym(r)$ has at most $r$ elements and the assertion follows.
\end{proof}

Let $T$ be an irreducible linear subgroup of $GL_{\F_p}(V)$. Write $\mathrm{Epi}(G,T)_T$ for the set of $T$-conjugacy classes of epimorphisms $G \to T$. The next lemma is essentially \cite[Lemma 7.2]{JP}.

\begin{lem}
	\label{epi}
	$|\mathrm{Epi}(G,T)|/|T| \leq |\mathrm{Epi}(G,T)_T| \leq |V||\mathrm{Epi}(G,T)|/|T|.$
\end{lem}
\begin{proof}
	Suppose $T \leq GL_{\F_p}(V)$. Then $Z(T) \leq C_{GL_{\F_p}(V)}(T) \leq \End_T(V)$, and $|\End_T(V)| \leq |V|$ because $V$ is an irreducible $T$-module. The result follows by the orbit-stabiliser theorem.
\end{proof}

We can now give a sufficient condition for UBERG, roughly analogous to \cite[Theorem 11.1, (4) $\Rightarrow$ (1)]{JP}. Let $c_4$ be the constant defined in \cite[Section 5]{JP} (see Section~\ref{sec:c4}). Also recall that, for a monolithic primitive group $L$ with minimal normal subgroup $N$, we say that $L$ is \emph{associated} with $A$ if $A$ is isomorphic to $N$. 
\begin{thm}
	\label{crownpowersUBERG}
Let $G$ be a $d$-generated profinite group. Suppose there is some $b$ such that, for all finite monolithic groups $L$ with non-abelian minimal normal subgroup $K$, if the crown-based power $L_k$ is a quotient of $G$ then $k\le l^{proj}(K)^b$. Then there is some $c$ such that, for $T$ an irreducible linear subgroup of $GL_{\F_p}(V)$, $|\Epi(G,T)| \leq |T||V|^c$ and $G$ has UBERG.
\end{thm}
\begin{proof}
	By \cite[Theorem 10.2, (4) $\Rightarrow$ (1)]{JP}, there is some $a$ such that, for any transitive group $Q$ of degree $k$ 
	\begin{equation}\label{eq:bound-epi}
		|\Epi(G,Q)| \leq |Q|a^k.
	\end{equation}
	
	Let $\dim_{\F_p}V = n$. Let $H$ be a subgroup of $T$ such that the representation of $T$ is induced from a primitive representation of $H$. Denote by $W$ a primitive $H$-module such that $V = \Ind_H^T W$. Let $P$ be the image of $H$ in $\End_{\F_p}(W)$ and $m = \dim_{\F_p} W$. Put $\tilde{K} = \mathrm{core} (H)$. Then $T/ \tilde{K}$ is a transitive group of degree $s = n/m$ and $T$ is a subgroup of $P \wr T/ \tilde{K}$.
	
	Case 1: Suppose that $|\tilde{K}| \leq |V|^{c_4}$.
	
	Then  $|\Epi(G,T/\tilde{K})| \leq |T/\tilde{K}|a^s$ by \eqref{eq:bound-epi}, and hence, by the hypothesis, $$|\Epi(G,T)| \leq |T/\tilde{K}|a^s |V|^{c_4d} \leq |T| |V|^{\log_p(a)+c_4d}.$$
	
	Case 2: Suppose that $|\tilde{K}| > |V|^{c_4}$.
	
	Since $|\tilde{K}| = |T|/|T/\tilde{K}| \leq |P|^s$ we have $|P| > |W|^{c_4}$ and so we can use \cite[Proposition 5.7]{JP}. We use the notation of that proposition. Denote $E(B)/Z(E(B))$ by $S$, where $E(B)$ is the layer of $B$ (see Section~\ref{sec:notation}). Let $K = (T \cap E(B)^s)'$. As in \cite[Propositions 6.1 and 7.1]{JP}, $K$ is normal in $T$ and is a subdirect product of $E(B)^s$, and $T$ acts transitively on the factors of $E(B)^s$; in particular, $K/Z(K) \cong S^r$, for some $r \leq s$, and $T/C_T(K)$ is the primitive group associated with $K/Z(K)$ (cf.\ Section~\ref{sec:crowns}).
	
	By \cite[Lemma 7.7]{GMP}, $$|\mathrm{Out}(S)| \leq 3 \log(l(S)) \leq 3 \log(l^{proj}(S)).$$ It follows that $T/KC_T(K)$ (which embeds into $\mathrm{Out}(S) \wr \Sym(r)$) is a transitive group of degree at most $3r \log(l^{proj}(S)) = 3 \log(l^{proj}(K/Z(K)))$, so by \eqref{eq:bound-epi}
	 \begin{align*}
	  \vert \Epi(G,T/KC_T(K)) \vert &\leq |T/KC_T(K)|a^{3 \log(l^{proj}(K/Z(K)))} \\
	  &= |T/KC_T(K)| l^{proj}(K/Z(K))^{3 \log(a)}.
	  \end{align*}
	  Fix some $\phi \in \Epi(G,T/KC_T(K))$, let $\phi_1, \ldots, \phi_j$ be its preimages in $\Epi(G,T/C_T(K))$, and note that $(\phi_1, \ldots, \phi_j): G \to \prod_{i=1}^j T/C_T(K)$ is a quotient map from $G$ to a crown power of $K/Z(K)$; writing $\mathrm{Aut}_1(T/C_T(K))$ for the group of automorphisms of $T/C_T(K)$ which induce the identity on $T/KC_T(K)$, we deduce using Lemma~\ref{lem:aut-1} that
	   \[ j \leq l^{proj}(K/Z(K))^b |\mathrm{Aut}_1(T/C_T(K))| \leq  l^{proj}(K/Z(K))^b r|\mathrm{Aut}(S)|^r\]
 and therefore 
 \begin{multline*}
 \vert \Epi(G,T/C_T(K))\vert  \leq \\ \leq |T/KC_T(K)| l^{proj}(K/Z(K))^{3 \log(a)} l^{proj}(K/Z(K))^b r|\mathrm{Aut}(S)|^r.
 \end{multline*}
  Again, by \cite[Lemma 7.7]{GMP}, $|\mathrm{Out}(S)| \leq l(S) \leq l^{proj}(S)$, so 
  \begin{multline*} |\Epi(G,T/C_T(K))| \leq \\ \leq |T/C_T(K)| l^{proj}(K/Z(K))^{3 \log(a) + b + 2} \leq |T/C_T(K)| |V|^{3 \log(a) + b + 2}
  \end{multline*}
   because $l^{proj}(K/Z(K)) \leq |V|$: indeed, $l^{proj}(K/Z(K)) = l^{proj}(S)^r$ and $|V| = |W|^s$, $s \geq r$, so it is enough to show that $l^{proj}(S) \leq |W|$, which holds because (by \cite[Proposition 5.7]{JP}) $W$ is an $E(B)$-module, defined over some finite extension of $\F_p$, on which $Z(E(B))$ acts by scalars, so it is a projective representation of $S = E(B)/Z(E(B))$.
	
	Finally, as in \cite[Proposition 7.1, Case (2)]{JP}, we have $|C_T(K)| \leq |V|^3$, so $|\Epi(G,T)| \leq |T| |V|^{3 \log(a) + b + 2 + 3d}$.
	
	To conclude that $G$ has UBERG, we apply Lemma \ref{epi} and \cite[Proposition 6.1]{JP}.
\end{proof}

\begin{rem}\label{rem:K=1}
Note that the case $K=1$ does not occur in the proof of the previous theorem (or in \cite[Proposition 7.1]{JP}), while it does in \cite[Proposition 6.1]{JP}, because here we are using $|\tilde{K}| > |V|^{c_4}$ instead of $|P| > |V|^{c_4}$. That is, in the notation of \cite[Proposition 6.1]{JP}, $\vert \tilde{K} \vert = |T \cap \tilde{N}| > |V|^{c_4}$. But $|\tilde{N}| = |N_{GL_{\F_p}(W)}(E(B))|^s$, and
\begin{align*}
	|N_{GL_{\F_p}(W)}(E(B))| &\leq |C_{GL_{\F_p}(W)}(E(B))||\Aut(E(B))| \\
	&\leq |W||\mathrm{Out}(E(B))||E(B)|
\end{align*}
by \cite[Proposition 5.7(5)]{JP}. Also $|\mathrm{Out}(E(B))| \leq |W|$ by \cite[Lemma 7.7]{GMP}. So $|\tilde{N}/N| \leq |W|^{2s} = |V|^2$. Therefore $|T \cap N| > |V|^{c_4-2} > |V|$ (note that $c_4\ge 3$), whereas $|Z(E(B))|^s \leq |W|^s = |V|$. So we cannot have $T \cap N \leq Z(E(B))$.
\end{rem}

For our necessary condition, we work parallel to \cite[Lemma 9.2]{JP}. While $l^{proj}(S)$ measures the size of the smallest projective representation of a non-abelian simple group $S$, we now need to measure the size of the smallest linear representation of $S$.

\begin{lem}
	\label{lem:faithful-factors}
	Let $F$ be finite field of characteristic $p$ and let $T$ be a finite group. Let $V$ be a faithful finite dimensional linear $F$-representation of $T$. Every element of $T$ which acts trivially on all composition factors of $V$ lies in $O_p(T)$. In particular, the smallest faithful linear representation of a finite group with a non-abelian unique minimal normal subgroup is irreducible.
\end{lem}
\begin{proof}
Let $\rho \colon T \to GL(V)$ be a finite dimensional faithful linear representation of $T$ defined over the finite field $F$.  Let $K$ be the common centraliser of all composition factors of $V$. Then, with respect to a suitable basis of $V$, the image $\rho(K)$ lies in the group of unimodular upper triangular matrices and thus is a $p$-group. As $\rho$ is faithful, we deduce that $K$ is a normal $p$-subgroup of $T$ and $K \subseteq O_p(T)$.

Suppose that $T$ has a unique minimal normal subgroup $K$, which is non-abelian. Since the center of $O_p(T)$ always contains a minimal normal subgroup, it follows that  $O_p(T) = \{1\}$ for all primes $p$. We deduce that every faithful representation admits a composition factor on which some $k \in K$ acts non-trivially. Since $K$ is the unique minimal normal subgroup, it follows that the representation on this composition factor is faithful.
\end{proof}

Given a non-abelian characteristically simple finite group $K = S^s$, $S$ simple, we define a new function $l^{lin}(K)$: if the smallest non-trivial projective representation $S \to GL(V)$ of $S$ has dimension $k$ over the field $F = \End_S(V)$, let $l^{lin}(K) = |F|^{ks}$ if this representation is (projectively equivalent to) a linear representation, and $l^{lin}(K) = |F|^{k^2s}$ if not. As justification for this definition when $s=1$, we offer the following lemma.

\begin{lem}
	\label{smallestrep}
	There is a constant $e$ such that, for any non-abelian simple group $S$, the smallest faithful irreducible linear representation of $S$ has size at least $l^{lin}(S)^e$ and at most $l^{lin}(S)$.
\end{lem}
\begin{proof}
	We first prove the lower bound. This is clear for the sporadic groups and the groups $S$ for which $l^{lin}(S) = l^{proj}(S)$ (which includes the alternating groups), leaving only simple groups of Lie type for which the natural module is not linear.

	For the exceptional groups of Lie type, the lower bound follows from \cite[Proposition 5.4.13]{KL}: there is some $e$ such that any non-trivial irreducible projective representation $S \to PGL_k(F)$ has size at least $|F|^{ek^2}$.
	
	For the classical simple groups of Lie type, the lower bound follows from \cite[Proposition 5.4.11]{KL}: apart from modules quasiequivalent to the natural module of dimension $d$, the smallest projective representation of these groups (and hence the smallest linear representation) has dimension at least $d(d-1)/2-2$ for large $d$. It is easy to see that the linear representations of $S$, as a subset of the projective representations of $S$, are invariant under quasiequivalence, so the result follows.
	
	For the upper bound (for all simple groups of Lie type), let $\rho \colon S \to PGL_k(F)$ be the smallest faithful projective representation of $S$. If $\rho$ lifts to a linear representation, then it follows from Lemma \ref{lem:faithful-factors} that the smallest linear faithful representation of $S$ is bounded from above by $|F|^{k}$. If $\rho$ does not lift to a linear representation, we compose $\rho$ with the adjoint representation $PGL_k(F) \to GL_{k^2}(F)$ to obtain a faithful linear representation for $S$ of size $|F|^{k^2}$. Again, the assertion follows from Lemma \ref{lem:faithful-factors}.
\end{proof}

\begin{lem}
	\label{linrep}
	There is some $c$ such that, for any non-abelian simple group $S$, $\mathrm{Aut}(S)$ has a faithful irreducible projective representation of size $\leq l^{proj}(S)^c$ and a faithful irreducible linear representation of size $\leq l^{lin}(S)^c$.
\end{lem}
\begin{proof}
	This is clear for sporadic and alternating groups; we assume $S$ is of Lie type. For this we can apply the description of $\mathrm{Aut}(S)$ from \cite[Section 2.5, 2.7]{GLS}. The automorphisms of $S$, and of its perfect central extensions $\tilde{S}$, are generated by: inner automorphisms, diagonal automorphisms, field automorphisms and graph automorphisms, as described there.
	
	Suppose $S$ is defined in degree $k$ over a field $F$ of characteristic $p$: then some perfect central extension $\tilde{S}$ of $S$ acts naturally and faithfully on $V$ of dimension $k$ over $F$, and $l^{proj}(S) = |F|^k$. We identify $\tilde{S}$ with its image in $GL_F(V)$. Fix a basis of $V$. From \cite[Section 2.5]{GLS}, we see that the diagonal automorphisms of $\tilde{S}$ are induced by conjugation by the subgroup $T$ of the diagonal elements of $GL_F(V)$ which normalise $\tilde{S}$. Now $\tilde{S}T \leq GL_F(V) \leq \Gamma L_F(V)$ (recall that $\Gamma L_F(V)$ is the semilinear group) is acted on by the Frobenius automorphisms $U \leq \Gamma L_F(V)$ of $F$ acting componentwise on the matrix entries. Thus $V$ becomes a crossed $\tilde{S}TU$-module over $F$, and also a crossed module for the universal central extension $R$ of $\tilde{S}TU$, via restriction.
	
	Also from \cite[Section 2.5]{GLS}, $\mathrm{Aut}(S)$ is the semidirect product of $R/Z(R)$ by a group $\Gamma$ of graph automorphisms of order at most $6$, and these automorphisms of $R/Z(R)$ extend to automorphisms of $R$. Form the semidirect product $R \rtimes \Gamma$ using this action, so that $(R \rtimes \Gamma)/Z(R) \cong \mathrm{Aut}(S)$. So $W = \Ind_R^{R \rtimes \Gamma}(V)$ is a faithful representation of $R \rtimes \Gamma$ of size $|F|^{k|\Gamma|}$. A simple quotient of this on which $S$ acts non-trivially must be an irreducible representation of $R \rtimes \Gamma$ with kernel $\leq Z(R)$, because $S$ is the unique minimal normal subgroup of $\mathrm{Aut}(S)$, so this gives a faithful irreducible projective crossed representation $X$, over some extension of $F$, of $\mathrm{Aut}(S)$, of size at most $|F|^{6k}$. The restriction of this to $\F_p$ gives a faithful irreducible projective representation of $\mathrm{Aut}(S)$ of size $l^{proj}(S)^6$, which is linear if $V$ is.
	
	For the case where $V$ is not linear, write $X^\ast$ for the dual module to $X$. The class of the cocycle of $X^*$ is the inverse of the cocycle of $X$ and so $X \otimes_{F} X^\ast$ is a crossed representation of $\mathrm{Aut}(S)$ with trivial cocycle (see Lemma \ref{lem:cocycles}).  Restriction of scalars to $\F_p$ provides us with a faithful linear representation of $\Aut(S)$ of size at most $|F|^{36k^2}$.
	Since $S$ is the unique minimal normal subgroup of $\Aut(S)$, it follows from Lemma \ref{lem:faithful-factors} that there is a faithful irreducible composition factor.
\end{proof}

We can now justify our definitions of $l^{proj}(K)$ and $l^{lin}(K)$, in the form of the following lemma. Suppose $T$ is a monolithic group with non-abelian minimal normal subgroup $K=S^s$, for some non-abelian simple group $S$.

\begin{prop}
	\label{primitiverep}
	Let $c$ be the constant from Lemma \ref{linrep}, and $e$ the constant from Lemma \ref{smallestrep}. The smallest size of a faithful irreducible linear (respectively, projective) representation of $T$ is:
	\begin{enumerate}[(i)]
		\item at least $l^{lin}(K)^e$ (respectively, $l^{proj}(K)$);
		\item at most $l^{lin}(K)^c$ (respectively, $l^{proj}(K)^c$).
	\end{enumerate}
\end{prop}
\begin{proof}
	We give a proof for linear representations; the proof for projective representations is similar.
	
	\begin{enumerate}[(i)]
		\item For the lower bound, suppose $V$ is a faithful irreducible $T$-module, and let $W$ be an irreducible summand of $V$ as a $K$-module. Some subset $A$ of the set $\{1,\ldots,s\}$ of (indices of the) copies of $S$ acts non-trivially on $W$; by \cite[Lemma 5.5.5]{KL}, $W$ is a tensor product (over the field $F = \End_K(W) = \End_{S_i}(X_i)$) of non-trivial modules $X_i$ for each of the $S_i$ in $T$. Each $X_i$ has size at least $l^{lin}(S)^e$, by Lemma \ref{smallestrep}, so (because $\dim_F(X_i) \geq 2$) $W$ has size $\geq l^{lin}(S)^{e|A|}$.
		
		Now each irreducible summand $W_j$ of $V$, considered as a $K$-module, has size $|W|$, thanks to Clifford theory. Moreover, for each $W_j$, the subset $A_j$ of the copies of $S$ acting non-trivially on it has size $\vert A\vert$. Each copy of $S$ acts non-trivially on $V$, so there is some $W_j$ on which it acts non-trivially; hence $\bigcup_j A_j = \{1,\ldots,s\}$, so the number of $W_j$ is at least $s/|A|$. Therefore $V$ has size $\geq |W|^{s/|A|} \geq l^{lin}(S)^{es}$.
	
		\item For the upper bound, pick one of the copies $S_1$ of $S$ in $K$. By Lemma~\ref{lem:KK}, $T$ embeds in $\Aut(S) \wr T/\tilde{K}$, where $\tilde{K} = \mathrm{core}(N_T(S_1))$. By the previous lemma, $\Aut(S)$ has a faithful irreducible representation $V$ of size $l^{lin}(S)^c$. By allowing $T/\tilde{K}$ to permute the copies of $V^s$, we get a faithful irreducible representation of $\Aut(S) \wr T/\tilde{K}$ of size $l^{lin}(K)^c$. The restriction of this representation to $L$ has a faithful composition factor by Lemma \ref{lem:faithful-factors}; this irreducible representation has size $\leq l^{lin}(K)^c$.
	\end{enumerate}
\end{proof}

\begin{thm}
	\label{necessary}
Let $G$ be a profinite group. Suppose, for all $b$, there is some monolithic group $L_b$ with non-abelian minimal normal subgroup $K_b$ such that some crown-based power $L_k$ of $L$ is isomorphic to a quotient of $G$ and $k> l^{lin}(K_b)^b$. Then $G$ does not have UBERG.

\end{thm}
\begin{proof}
By Proposition \ref{primitiverep}, each $L_b$ has a faithful irreducible representation of size $\leq l^{lin}(K_b)^c$. On the other hand, from projecting onto the factors of the crown power, we have more than $l^{lin}(K_b)^b$ epimorphisms $G \to L_b$ with different kernels, and thus more than $l^{lin}(K_b)^b$ $G$-modules of size $\leq l^{lin}(K_b)^c$. Since this is true for all $b$, we conclude that $G$ does not have UBERG.
\end{proof}

We can show that the gap between our necessary and sufficient conditions for UBERG is the best one can do by considering only crown-based power quotients of $G$ with non-abelian socle.

\begin{thm}
	\label{cover/quotient}
	Let $G$ be a $d$-generated profinite group. The universal Frattini cover $\tilde{G}$ has UBERG if and only if there is some $b$ such that, for all monolithic groups $L$ with non-abelian minimal normal subgroup $K$, the size of a crown-based power $L_k$ of $L$ occurring as a quotient of $G$ is $k \leq l^{proj}(K)^b$. 
	\end{thm}
\begin{proof}
	As every crown-based power $L_k$ of a monolithic group $L$ with non-abelian socle is Frattini-free, 
	$\Epi(G,L_k)$ is canonically isomorphic to $\Epi(\tilde{G},L_k)$. So the `if' part follows immediately.
	
	Suppose for all $b$ there is some monolithic group $L$ with non-abelian minimal normal subgroup $K$, such that the size of a crown-based power of $L$ occurring as a quotient of $G$ is $> l^{proj}(K)^b$. By Proposition~\ref{primitiverep}, $L$ has a faithful irreducible projective representation of size $\leq l^{proj}(K)^c$, so as in Theorem \ref{necessary}, the crown-based power of $L$ gives more than $l^{proj}(K)^b$ different projective representations of this size. On the other hand, $\tilde{G}$ is a projective profinite group and thus every projective representation of $\tilde{G}$ lifts to a linear representation. Since this holds for all $b$, this shows $\tilde{G}$ does not have UBERG.
\end{proof}

For later reference, we say that a profinite group has proj-UBERG if it satisfies the condition on crown-based powers in the previous theorem.

As an application for our sufficient condition, we give a slight sharpening of the known result that PFG implies UBERG.

\begin{cor}
	\label{projUBERG}
	For finitely generated profinite groups:
	\[ \text{PFG}  \Longrightarrow \text{proj-UBERG} \Longrightarrow \text{UBERG}.
	\]
\end{cor}
\begin{proof}
	Suppose $G$ is PFG. By \cite[Theorem 11.1]{JP}, for monolithic group $L$ associated with non-abelian minimal normal subgroup $K$, there is some $a$ such that the size of a crown-based power of $L$ occurring as a quotient of $G/N$ is $l(K)^a$, which is less than $l^{proj}(K)^a$ by Lemma \ref{lengths}. The second implication holds by Theorem \ref{crownpowersUBERG}.
\end{proof}

\section{Infinite products of finite groups with UBERG}\label{sec:inf_prod_fin_grps}
In this section we give a criterion which allows to verify that an infinite product of finite groups has UBERG. Our approach is rather direct and relies on elementary calculations. The method can be used to reprove results of Damian \cite[Example 4.5]{Damian}, which rely on the machinery of \cite{JP}. We use our method to construct a UBERG group which is not finitely generated. In addition, we use it to show that groups with UBERG are not closed under extensions; indeed, we construct a product of special linear groups which is procyclic-by-UBERG but which does not have UBERG. We show that both normal subgroup and quotient are finitely presented and hence PFR by \cite[Theorem A]{KV}. This answers in the negative an open question from \cite{KV}. The next section, however, contains conditions under which extensions of groups with UBERG do have UBERG.

\subsection{A criterion for UBERG}
Here a \emph{family of power series} $(S_{q})_{q}$ is a sequence of power series with $S_q \in \Z\bra X \ket$ for every prime power $q$.
\begin{define}
A family of power series $(S_{q})_{q}$ is \emph{uniformly bounded}
if there are constants $c,B > 0$ such that
\[
	S_q(q^{-c}) \leq B \cdot q^c
\]
for all prime powers $q$.
\end{define}
Let $G$ be a profinite group. We write $r^{\ast}(G,F,n)$ to denote the number of absolutely irreducible representations of $G$ defined over $F$. Assume that $r^{\ast}(G,F,n) < \infty$ for all $n \in \mathbb{N}$ and all finite fields $F$. For instance, all finitely generated profinite groups have this property.
For every finite field $\bbF_q$, we consider the power series
\[
	S^{\ast}_q(G) = \sum_{n=1}^\infty r^{\ast}(G,\bbF_q,n) X^{n-1} \in \Z\bra{X}\ket.
\]
\begin{lem}
Let $G$ be a profinite group with the property $r^{\ast}(G,F,n) < \infty$ for all $n \in \mathbb{N}$ and all finite fields $F$.
The group $G$ has UBERG if and only if the family of power series $S_q^\ast(G)$
is uniformly bounded.
\end{lem}
\begin{proof}
Assume $G$ has UBERG, then there is a constant $c> 0$ such that
$r^\ast(G,\bbF_q,n) \leq q^{nc}$ for all $n \in \mathbb{N}$ and all finite fields $\bbF_q$; see \cite[Lemma 6.8]{KV}. In particular,
\begin{align*}
	S_q^\ast(G)(q^{-c-1}) &= \sum_{n=1}^\infty r^{\ast}(G,\bbF_q,n) q^{-(c+1)(n-1)} \leq \sum_{n=1}^\infty  q^{-n +c +1} \\
	&= q^{c+1} \frac{q^{-1}}{1-q^{-1}} \leq q^{c+1}
\end{align*}
and $(S_q^\ast(G))_q$ is uniformly bounded.

Conversely, assume that $S_q^\ast(G)$ is uniformly bounded and $S_q^\ast(G)(q^{-c}) \leq q^{c}B$.
Then in particular
$r^\ast(G,\bbF_q,n) q^{-c(n-1)} \leq q^c B$ and (since $c$ and $B$ are independent of $q$) this shows that $G$ has UBERG.
\end{proof}

\begin{lem}\label{lem:uberg-criterion}
Let $(G_i)_{i \in \mathbb{N}}$ be a family of profinite groups such that the power series 
$S_q^{\ast}(G_i)$ are defined.
If the family of power series
\[
	S_{q} = \prod_{i\in \mathbb{N}} S_q^{\ast}(G_i) 
\]
is well-defined and uniformly bounded, then $G =\prod_{i} G_i$ has UBERG.
\end{lem}
\begin{proof}
We claim that for every real number $x \in [0,1)$ such that $S_q$ converges at $x$,
we have
\[
	S_q^{\ast}(G)(x) \leq S_q(x).
\]
This suffices to conclude that $G = \prod_{i\in \mathbb{N}} G_i$ has UBERG.

The absolutely irreducible representations of $G$ are outer tensor products
$\bigotimes V_i$ of absolutely irreducible representations of the factors $G_i$, such that almost all $V_i$ are trivial. In particular, we obtain
\[
	r^{\ast}(G,F,n) = \sum_{\underline{d} \in D_n} \prod_{i} r^{\ast}(G_i,F,d_i)
\]
where the sum runs over the set $D_n$ all sequences $\underline{d} = (d_1,d_2,\dots)$ such that almost all $d_i$ equal $1$ and $\prod_{i}d_i = n$.
We observe that for all such sequences $\underline{d}$ we have
\[
	\sum_{i=1}^\infty (d_i - 1) \leq (\prod_{i=1}^\infty d_i)-1 = n-1
\]
and so, for all $x \in [0,1)$, we have $x^{n-1} \le \prod_i x^{d_i-1}$. Therefore
\begin{align*}
	S_q^{\ast}(G)(x) =& \sum_{n=1}^\infty \sum_{\underline{d} \in D_n} x^{n-1}\prod_{i} r^{\ast}(G_i,\bbF_q,d_i) \\
	\leq& \sum_{n=1}^\infty \sum_{\underline{d} \in D_n} \prod_{i} x^{d_i - 1}r^{\ast}(G_i,\bbF_q,d_i) 
	= S_{q}(x).
\end{align*}
\end{proof}

\subsection{An infinitely generated UBERG group}\label{subsec:inf_gen_UBERG}
In this section we give the first example of a profinite group which cannot be finitely generated but has UBERG. This group $G$ is of the form
\[
  G= \prod_{i=1}^\infty G_{p_i,n_i}^\ast
\]
for certain finite metabelian group $G_{p_i,n_i}^\ast$. We begin by describing the building blocks.

Let $p > 3$ be a prime with $p \equiv 3 \bmod 4$. Let $C_p$ be a cyclic group of order $p$. Define $\mathrm{Aut}(C_p)^\circ$  to be the subgroup of $\mathrm{Aut}(C_p)$ which consists of automorphisms of odd order.
The group $\mathrm{Aut}(C_p)^\circ$ is cyclic of order $\eta(p)=\frac{p-1}{2}$ (since $p \equiv 3 \bmod 4$). 
Consider the group 
\[ G_{p,n}^\ast = C_p^n \rtimes \mathrm{Aut}(C_p)^\circ\]
with $\mathrm{Aut}(C_p)^\circ$ acting diagonally on $C_p^n$.
We observe that $\mathrm{Aut}(C_p)^\circ$ is non-trivial (since $p>3$) and it is the abelianisation of $G^\ast_{p,n}$.

\begin{lem}
Let $G_{p,n}^\ast = C_p^n \rtimes \mathrm{Aut}(G_p)^\circ$. Then  $\mathrm{d}(G_{p,n}^\ast) \ge n$.
\end{lem}
\begin{proof}
 Let $d=\mathrm{d}(G_{p,n}^\ast)$ and let $(v_1,\alpha_1),\ldots,(v_d,\alpha_d)$ be a set of generators. Since the action of $\mathrm{Aut}(G_p)^\circ$ on $C_p^n$ is diagonal, $v_1,\ldots,v_d$ have to generate $C_p^n$ and hence $d\ge n$.
\end{proof}

 We take a closer look at the representations of these groups. Let $\ell$ be a fixed prime. We start by studying the irreducible representations of $G_{p,n}^\ast$ over an algebraically closed field $K$ of characteristic $\ell$. 
The one-dimensional irreducible representations of $G_{p,n}^\ast$ are exactly the ones which factor through the abelianisation, i.e., through $\mathrm{Aut}(C_p)^\circ$; there is one of these for each $\eta(p)$-th root of unity in $K$. In particular, if $F$ is any field finite field, we obtain
\[
   r^\ast(G_{p,n}^\ast, F,1) = \gcd(\eta(p),|F|-1).
\]

If the characteristic $\ell = p$, then every irreducible representation $V$ of $G_{p,n}^\ast$ over $K$ is one-dimensional. Indeed,
$V|_{C^n_p}$ is trivial, since it is semi-simple and the trivial representation is the only irreducible representation of a $p$-group in characteristic $p$.

Assume that $\ell \neq p$. 
Let $(\rho,V)$ be an irreducible representation of $G_{p,n}^\ast$ over $K$ that does not factor through $\mathrm{Aut}(C_p)^\circ$. Restriction to $C_p^n$ gives
\[
  V_{|C_p^n} = e \cdot \bigoplus_{\alpha \in \mathrm{Aut}(C_p)^\circ} \hspace{-0.3cm}\phantom{U}^\alpha  U
\]
with $U$ an irreducible representation of $C_p^n$ over $K$. Since $(\rho, V)$ is not one-dimensional, $U$ is not the trivial representation; $U$ is one-dimensional, and the action of a non-trivial $\alpha \in \mathrm{Aut}(C_p)^\circ$ has no non-trivial fixed points in $C_p^n$. It follows that $^\alpha U$ is not isomorphic to $U$ for any $\alpha \neq \mathrm{id}$, and hence the inertial subgroup is $C_p^n$ itself. Therefore $e=1$ (see \cite[Theorem 2.2.2(iii), p.84]{Karpilovsky}).

It follows that $V=\mathrm{Ind}_{C_p^n}^{G_{p,n}^\ast}(U)$ and that
\[
  \dim_{K}(V) = \vert G_{p,n}^\ast : C_p^n \vert \cdot \dim_{\mathbb{F}_\ell} U = \vert \mathrm{Aut}(C_p)^\circ \vert = \eta(p).
\]
Now note that the non-trivial representations $(\rho,V)$'s as above correspond to non-trivial $\mathrm{Aut}(C_p)^\circ$-orbits on $\mathrm{Irr}(C_p^n, K)$. In particular, for every field $F$ of characteristic $\ell$ we find 
\[
  r^\ast(G_{p,n}^\ast, F, \eta(p)) \leq p^n - 1. 
\]
We have proved the following:
\begin{lem}
	\label{lem:bound-example-infi}
For all $x \in [0,\infty)$ and all prime powers $q$, we have
\[
	S_q^\ast(G_{p,n}^\ast)(x) \leq \gcd(\eta(p),q-1) + (p^n-1)x^{\eta(p)-1}
\]
\end{lem}

Now we are able to construct the infinitely generated UBERG group $G$. Let $p_1$, $p_2$, $\ldots$ be an increasing sequence of primes $p>3$ satisfying:
\begin{enumerate}[(i)]
 \item $p_i \equiv 3 \bmod 4$, and
 \item $\frac{p_1 -1}{2}, \frac{p_2 -1}{2},\ldots$ are pairwise coprime. 
\end{enumerate}
Recall that $\eta(p_i) = \frac{p_i -1}{2}$.
Additionally, pick an increasing and unbounded sequence of integers $n_i$ such that 
\[
	p_i^{n_i} - 1 \leq 2^{\eta(p_i)}.
\]
\begin{lem}\label{lem:UBERG_inf_gen}
For $p_i$ and $n_i$ as above, the group
\[
  G= \prod_{i=1}^\infty G_{p_i,n_i}^\ast.
\]
has UBERG and does not have type $\FP_1$, so cannot be finitely generated and does not have type $\PFP_1$.
\end{lem}
\begin{proof}
Clearly, $G$ is not finitely generated, as $\mathrm{d}(G) \ge \mathrm{d}(G_{p_i,n_i}^\ast) \ge n_i$ and the integers $n_i$ tend to infinity. Since $G$ is soluble, by \cite[Corollary 2.4]{Damian} and the remark after it, it must not have type $\FP_1$.

For every prime power $q$ we have (by Lemma \ref{lem:bound-example-infi})
\begin{align*}
	\prod_{i=1}^\infty S^\ast_q(G_{p_i,n_i}^\ast)(q^{-2}) &\leq \prod_{i=1}^\infty\bigl( \gcd(\eta(p_i),q-1) + (p_i^{n_i}-1)q^{-2(\eta(p_i)-1)} \bigr)\\
	&\leq \prod_{i=1}^\infty\bigl(\gcd(\eta(p_i),q-1) + q^{\eta(p_i)-2(\eta(p_i)-1)}\bigr)\\
	&\leq (q-1)\prod_{i=1}^\infty\bigl( 1 + q^{-\eta(p_i)+2}\bigr)\\
	&\leq (q-1)\prod_{i=1}^\infty(1 + q^{-i}) \\
	&\leq (q-1) \exp(\sum_{i=1}^\infty q^{-i}) \leq (q-1)e.
\end{align*}
So the family of power series $\prod_{i=1}^\infty S^\ast_q(G_{p_i,n_i}^\ast)$ is uniformly bounded and, using Lemma \ref{lem:uberg-criterion}, we deduce that $G$ has UBERG.
\end{proof}

\subsection{A non-UBERG procyclic-by-UBERG group}
\label{sec:UBERGbyPUBERG_count}
For $m \geq 1$ we consider the direct product $\SL_n(\bbF_q)^m$ and the factor group $G_n(q,m) = \SL_n(\bbF_q)^m/C_n(q,m)$ where $C_n(q,m)$ is the image of the diagonal embedding of the centre $C \leq \SL_n(\bbF_q)$. The centre $C$ consists of scalar matrices $\lambda I_n$, where $\lambda \in \bbF_q^\times$ is an $n$-th root of unity. In particular, the group $C_n(q,m)$ is cyclic of order $|C_n(q,m)| = \gcd(n,q-1)$.

Let $(n_i)_{i\in \mathbb{N}}$ be an increasing sequence of pairwise coprime integers $\geq 12$ and let $q_i = p_i^{k_i}$ be a sequence of prime powers for pairwise distinct primes $p_i \geq 5$. We may assume that $\gcd(n_i,q_i-1) > 1$ for every $i$. We define the sequence $(m_i)_{i \in \mathbb{N}}$ as
$m_i = q_i^{\lfloor n_i^{3/2} \rfloor}$. Observe that $m_i$ grows faster than $q_i^{cn_i}$ for every $c > 0$ but slower than $q_i^{n_i^2}$
as $i$ tends to $\infty$.

We consider the profinite groups
\[
	G = \prod_{i \in \mathbb{N}} \SL_{n_i}(\bbF_{q_i})^{m_i}
\]
and
\[
	H = \prod_{i \in \mathbb{N}} G_{n_i}(q_i,m_i).
\]
We note that $H \cong G/Z$ where
$Z = \prod_{i} C_{n_i}(q_i,m_i)$. The group $Z$ is a procyclic group, since $C_{n_i}(q_i,m_i)$ is a cyclic group of order $\gcd(n_i,q_i-1)$ and these orders are pairwise coprime (indeed, the integers $n_i$ were chosen to be pairwise coprime).

\begin{thm}
	\label{thm:example-procyclic-by-uberg}
The groups $G$ and $H$ are 2-generated and finitely presented. The group $H$ has UBERG, but $G$ does not have UBERG.
\end{thm}
\subsubsection{First part of proof of Theorem \ref{thm:example-procyclic-by-uberg}}
\begin{proof}[Proof that $G$ does not have UBERG]
It is clear that the group $\SL_{n_i}(\bbF_{q_i})^{m_i}$ has at least
$m_i$ absolutely irreducible representations of degree $n_i$ over $\bbF_{q_i}$, i.e.\
\[ 
	r^{\ast}_{n_i}(G,\bbF_{q_i}) \geq m_i = q_i^{\lfloor n_i^{3/2} \rfloor}.
\]
We observe that $q_i^{\lfloor n_i^{3/2} \rfloor}$ grows faster than $q_i^{cn_i}$ for every $c > 0$ and we conclude that $G$ does not have UBERG.
\end{proof}
\begin{proof}[Proof that $G$ and $H$ are $2$-generated]
It suffices to prove that $G$ is $2$-generated.
We note that, since $G$ is perfect, the group $G$ is 
$2$-generated if and only if 
\[
	G/Z(G) \cong \prod_{i} \mathrm{PSL}_{n_i}(\bbF_{q_i})^{m_i}
	\]
is $2$-generated. Since the finite simple groups $\mathrm{PSL}_{n_i}(\bbF_{q_i})$ are pairwise non-isomorphic, this group is $2$-generated exactly if each block $\mathrm{PSL}_{n_i}(\bbF_{q_i})^{m_i}$ is $2$-generated.
By \cite[Theorem 1.3]{MQRD2013}, this is the case if 
\[
	m_i < \frac{|\PSL_{n_i}(\bbF_{q_i})}{\log(|\PSL_{n_i}(\bbF_{q_i})|)}
\]
We will establish the inequality $q^{n^{3/2}} \leq \sqrt{|\PSL_{n}(\bbF_{q})|}$ for all $n\geq 9$ and all $q$; this implies the required inequality, since $\sqrt{x} \geq \log(x)$ for all $x \geq 1$.
We observe that $|\PSL_n(\bbF_q)| \geq q^{n^2-n-2}$, so the we only have to prove that
\[
	2n^{3/2} < n^2 -n -2
\]
holds for all $n \geq 9$. This follows easily by induction.
\end{proof}
\begin{proof}[Proof that $G$ and $H$ are finitely presented]
It suffices to prove that $G$ is finitely presented, since $H = G/Z$ where $Z$ is procyclic.
We already know that $G$ is finitely generated, so by \cite[Theorem 0.3]{Lubotzky} it is sufficient to show that there is a constant $C>0$ such that
\[
	\dim_{\bbF_\ell} H^2(G,V) \leq C \dim_{\bbF_\ell} V
\]
for every prime $\ell$ and every finite irreducible representation $V$ of $G$ over $\bbF_\ell$. Extending scalars, we see that it is sufficient to establish such an upper bound for the second cohomology for every
absolutely irreducible representation $V$ over some finite field.

Let $V$ be an absolutely irreducible representation of $G$ over some finite field $\bbF$.
As $V$ is continuous, it factors over a finite product of components of $G$, i.e., we can write $G = G_1 \times \dots, G_t \times K$ where $K$ acts trivially on $V$ and each $G_j$ is of the form $\SL_{n_i}(\bbF_{q_i})$ (for some $i$) and acts non-trivially on $V$.
As $V$ is absolutely irreducible, it decomposes as a tensor product
\[
	V = W \otimes \bigotimes_{j=1}^t V_j 
\]
where each $V_j$ is an absolutely irreducible representation of $G_j$ and $W$ is the trivial $1$-dimensional representation of $K$.
By K\"unneth's formula we have
\[
	H^2(G,V) = \bigoplus_{f+ d_1+ \dots +d_t = 2} H^f(K,W) \otimes \bigotimes H^{d_j}(G_j,V_j) 
\]
Since each $\SL_{n_i}(q_i)$ is the universal central extension of $\PSL_{n_i}(q_i)$ (see \cite[Corollary 2 in \S 7]{Steinberg}, recall $p_i \geq 5$) the group $K$ is a universal central extension of some product of projective special linear groups. We deduce that $H^2(K,W) = 0$.
Since each $V_j$ is non-trivial and irreducible, we have $H^0(G_j,V_j) = 0$.
In particular, for $t \geq 3$ at least one $d_j = 0$ and we have $H^2(G,V)=0$.
Suppose that $t =1$, then 
\[
	 \dim_{\bbF} H^2(G,V) \leq  \dim_{\bbF} H^2(\SL_{n_i}(\bbF_{q_i}), V) \leq 8.5  \dim_{\bbF} V
\] 
by \cite[Theorem 8.4]{GKKL}.
If $t = 2$, then, since the special linear groups are $2$-generated, we have
$\dim_{\bbF}H^1(G_j,V_j) \leq 2 \dim_{\bbF}V$ and
\[
	\dim_{\bbF} H^2(G,V) \leq 4 \dim_{\bbF} (V_1)\dim_{\bbF}(V_2) = 4 \dim_{\bbF} V.\qedhere
\]
\end{proof}

It remains to show that $H$ has UBERG. This is more difficult and we need some more information on the representation theory of special linear groups.

\subsubsection{Representation theory of $\SL_n(\bbF_q)$.}
\begin{define}
Let $G$ be a finite group and let $F$ be a field. The \emph{minimal degree} $\md(G,F)$ is the degree of the smallest non-trivial absolutely irreducible representation of $G$ over $F$.
\end{define}

Let $p$ be an odd prime number and let $q$ be a power of $p$. 
Let $n \geq 3$ denote a natural number.
We consider the special linear group $\SL_n(\bbF_q)$. 
We are interested in absolutely irreducible representations of $\SL_n(\bbF_q)$ over finite fields $\bbF_{\ell^j}$ (where $\ell$ is some prime number). We will use the following facts:

\medskip

(a) \emph{In defining characteristic ($\ell = p$)} then $\md(\SL_n(\bbF_q),\bbF_{\ell^j}) = n$. The standard representation of $\SL_n(\bbF_q)$ on $\bbF_q^n$ and its dual are the only absolutely irreducible representations of minimal degree. For $n \geq 12$ every other non-trivial irreducible representation has degree $\geq n(n-1)/2$; see \cite[Theorem 5.1]{Luebeck}.

\medskip

(b) \emph{In non-defining characteristic ($\ell \neq p$)} for $n \geq 5$ the minimal degree satisfies $\md(\SL_n(\bbF_q),\bbF_{\ell^j}) \geq \frac{q^{n}-1}{q-1}-n \geq q^{n-1}$; see \cite[\S 2, Proposition]{Seitz-Zalesskii}

\medskip

(c) The number of absolutely irreducible representations of $\SL_n(\bbF_q)$ over any field is at most the number of conjugacy classes in $\SL_n(\bbF_q)$ which is smaller than $28 q^{n-1} \leq q^{n+3}$; see \cite[Theorem 1.1]{FulmanGuralnick}.

\medskip

This also allows us to bound the minimal degrees for the factor group $G_n(q,m) = \SL_n(\bbF_q)^m/C_n(q,m)$ where $C_n(q,m)$ is the image of the diagonal embedding of the centre $C \leq \SL_n(\bbF_q)$.
\begin{lem}\label{lem:minimal-degrees}
In non-defining characteristic $\ell \neq p$ the minimal degrees of $\SL_n(\bbF_q)^m$ and $G_n(q,m)$ are bounded from below by $q^{n-1}$.
In defining characteristic $\ell = p$ the following hold
\begin{enumerate}[(i)]
	\item $\md(\SL_n(\bbF_q)^m,\bbF_{\ell^j}) = n$ and there are exactly $2m$ representations of minimal degree.
	\item if $\gcd(n,q-1) > 1$, then $\md(G_n(q,m),\bbF_{\ell^j}) \geq n(n-1)/2$. 
\end{enumerate}
\end{lem}
\begin{proof}
In non-defining characteristic $\ell \neq p$, the absolutely irreducible representations of $\SL_n(\bbF_q)^m$ over $\bbF_\ell^j$ are tensor products $V_1 \otimes_{\bbF_{\ell^j}} V_2 \otimes_{\bbF_{\ell^j}} \dots \otimes_{\bbF_{\ell^j}} V_m$ of absolutely irreducible representations of $\SL_n(\bbF_q)$; see \cite{Fein}. Hence $\md(\SL_n(\bbF_q)^m,\bbF_{\ell^j}) = \md(\SL_n(\bbF_q),\bbF_{\ell^j})$ and the representations of minimal degree are 
of the form $V_1 \otimes_{\bbF_{\ell^j}} V_2 \otimes_{\bbF_{\ell^j}} \dots \otimes_{\bbF_{\ell^j}} V_m$ with exactly one non-trivial $V_i$ which is of minimal degree.

Since every representation of $G_n(q,m)$ lifts to representation of $\SL_n(\bbF_q)^m$ we have $\md(G_n(q,m),\bbF_{\ell^j}) \geq \md(\SL_n(\bbF_q)^m,\bbF_{\ell^j})$.

It remains to consider the defining characteristic case $\ell = p$. Let
$V = V_1 \otimes_{\bbF_{\ell^j}} V_2 \otimes_{\bbF_{\ell^j}} \dots \otimes_{\bbF_{\ell^j}} V_m$ be an absolutely irreducible representation of $\SL_n(\bbF_q)^m$ and let
$\omega_i\colon C \to \bbF_q^\times$ denote the central character of $V_i$. The representation
$V$ factors through $G_n(q,m)$ if and only if $\omega_1\omega_2\cdots\omega_m = 1$.

Assume that $\gcd(n,q-1) > 1$, so that the centre $C \leq \SL_n(\bbF_q)$ is non-trivial. In particular, the standard representation of $\SL_n(\bbF_q)$ on $\bbF_q^n$ and its dual have a non-trivial central character.
In particular, an absolutely irreducible representation $V$ of $\SL_n(\bbF_q)^m$ which factors through $G_n(q,m)$ and contains a tensor factor $V_i$ with $\dim(V_i) = n $ involves another non-trivial tensor factor. We deduce that $\dim(V) \geq \md(\SL_n(\bbF_q),\bbF_{\ell^j})^2 = n^2$.
We observe that any other non-trivial absolutely irreducible representation $V$ of $\SL_n(\bbF_q)^m$ which factors through $G_n(q,m)$ contains 
some $V_i$ of degree at least $n(n-1)/2$. This completes the proof.
\end{proof}
This result implies bounds on the number of irreducible representations of bounded degree, which are conveniently expressed in terms of the family of power series $S^{\ast}(G_n(q,m))$.
\begin{lem}\label{lem:bounds-series}
In non-defining characteristic $\ell \neq p$, we have
\[
	S_{\ell^j}^{\ast}(G_n(q,m))(x) \leq (1+q^{n+3}x^{q^{n-1}-1})^m
\]
for all $j$ and all $x \in (0,1)$.

Assume that $\gcd(n,q-1) > 1$. 
Then in defining characteristic $\ell = p$ we have
\[
	S_{p^j}^\ast(G_n(q,m))(x) \leq   1 +  \sum_{k=2}^{\infty} m^k q^{k n^2}x^{n^k}
\]
for all $x \in (0,1)$ and $j$.
\end{lem}
\begin{proof}
Assume that $\ell \neq p$. Let $x \in [0,1)$, then we obtain
\[
	S_{\ell^j}^{\ast}(G_n(q,m))(x) \leq S_{\ell^j}^{\ast}(\SL_n(\bbF_q)^m)(x) \leq S_{\ell^j}^{\ast}(\SL_n(\bbF_q))(x)^m
	\]
as in the proof of Lemma \ref{lem:uberg-criterion}.
By Fact 2 the minimal degree $\md(\SL_n(\bbF_q), \bbF_{\ell^j}) \geq q^{n-1}$ and by Fact 3 there are at most $q^{n+3}$ non-trivial representations, hence 
\[
	S^\ast(\SL_n(\bbF_q),\bbF_{\ell^j})(x) \leq 1 + q^{n+3}x^{q^{n-1}-1}
\] 
for all $x \in [0,1)$.

Now consider the defining characteristic case $\ell = q$.
The absolutely irreducible representations of $G_n(q,m)$ are the representations of $\SL_n(\bbF_q)^m$ with trivial restriction to $C_n(q,m)$. We know from the proof of Lemma \ref{lem:minimal-degrees} that the representations of degree $n$ don't factor through $G_n(q,m)$. Using Fact 3, we obtain
\begin{align*}
	S_{p^j}^\ast(G_n(q,m))(x) &\leq 1 +  \sum_{k=2}^{\infty} \binom{m}{k} (q^{n+3})^k x^{n^k-1} \\
	&\leq  1 +  \sum_{k=2}^{\infty} m^k q^{(n+3)k}x^{n^k-1}
\end{align*}
for all $x \leq 1$.
\end{proof}

\subsubsection{Second part of proof of Theorem \ref{thm:example-procyclic-by-uberg}}
\begin{proof}[Proof that $H$ has UBERG]
We want to apply Lemma \ref{lem:uberg-criterion}.
To this end we show that the family of power series $S_t = \prod_{i=1}^\infty S_t^\ast(G_{n_i}(q_i,m_i))$ where $t$ varies over all prime powers is uniformly bounded. We claim that the constant $c=2$ works.
 
Fix a prime power $t = \ell^j$.
Suppose the $\ell \neq p_i$ for all $i$, then by Lemma \ref{lem:bounds-series} we obtain
\begin{align*}
	S_{{\ell^j}}(\ell^{-2j}) 
		&\leq \prod_{i=1}^\infty(1+q_i^{n_i+3}\ell^{-2j(q_i^{n_i-1}-1)})^{m_i}\\
		&\leq \exp(\sum_{i=1}^\infty m_i q_i^{n_i+3}\ell^{-2j(q_i^{n_i-1}-1)}).
\end{align*}
A short calculation yields
\begin{align*}
  \sum_{i=1}^\infty m_i q_i^{n_i+3}\ell^{-2j(q_i^{n_i-1}-1)}
   &\leq 
  \sum_{i=1}^\infty q_i^{n_i^{3/2} + n_i+3}\ell^{-2j(q_i^{n_i-1}-1)}\\ 
  &\leq
  \sum_{i=1}^\infty 2^{\log_2(q_i)(n_i^{3/2} + n_i+3) - 2(q_i^{n_i-1}-1)}\\
  &\leq
  C+ \sum_{i=1}^\infty 2^{-i} = C + 1
\end{align*}
since obviously $\log_2(q_i)(n_i^{3/2} + n_i+3) - 2(q_i^{n_i-1}-1) < -i$ for all large $i$.
In particular, this series converges and is bounded above independently of $\ell^j$.

If $\ell = p_i$ but $\ell^j < q$, then we get the same result, since all absolutely irreducible representations of $\SL_n(\bbF_q)$ are defined over $\bbF_q$.
Assume finally that $\ell = p_i$ and $\ell^j \geq q$.
In this case, there is a new factor of the form 
\begin{align*}
	S_{p_i^j}^\ast(G_{n_i}(q_i,m_i))(\ell^{-2j}) &\leq  1 +  \sum_{k=2}^{\infty} m_i^k q_i^{k (n_i+3)}p_i^{-2j(n_i^k-1)} \\
	&\leq 1 +  \sum_{k=2}^{\infty} q_i^{2kn_i^{3/2}}q_i^{-2(n_i^k-1)}
	\leq 1 +  \sum_{k=2}^{\infty} q_i^{-2n_i^k(1-n_i^{-k}-kn_i^{3/2-k})}\\
	&\leq 1 + \sum_{k=2}^{\infty} q_i^{-(2/3)n_i^k}
\leq 2
\end{align*}
where we use the rough estimate $1-n_i^{-k}-kn_i^{3/2-k} \geq 1/3$ (for $n_i \geq 12$ and $k \geq 2$).
This factor is independent of $\ell$ and $j$ and so the claim follows.
\end{proof}
\section{UBERG-by-UBERG groups}\label{sec:UBERG-by-UBERG}
In this section we study conditions under which extensions of groups with UBERG have UBERG.
We show that split UBERG-by-UBERG groups, and UBERG-by-(finitely generated proj-UBERG) groups, have UBERG.

We proceed using the machinery of Clifford theory, for which our main reference is \cite{Karpilovsky}. Since any finite $G$-module is fixed pointwise by some open normal subgroup of $G$, we may think of such a module as the restriction of a module for some finite quotient of $G$. 

For a field $F$ and an $F$-algebra $A$, we write 
$r(A,F,n) = |\Irr(A,F,n)|$ to denote the number (isomorphism classes) of simple $A$-modules of $F$-dimension $n$ and $R(A,F,n) = \sum_{i=1}^n r(A,F,n)$. Clearly, $r(A,F,n) \leq R(A,F,n)$. If $E/F$ is a finite field extension such that $E \subseteq A$ is a subfield, we have $R(A,E,n) = R(A,F,n[E:F])$.
For a profinite group $G$, we write $r(G,F,n) = r(F\bra G \ket,F,n)$ and
$R(G,F,n) = R(F\bra G \ket,F,n)$.
Recall by \cite[Theorem A]{KV} that a profinite group $G$ has UBERG if and only if there is some constant $c$ such that $r(G,\F_p,n) \leq p^{cn}$ for all $n$ and $p$.

\begin{lem}\label{rvsR}
Let $F$ be a finite field, 
let $A$ be an $F$-algebra and let $c >0$. Then
\[
	(\forall n\in \mathbb{N})\; r(A,F,n) \leq |F|^{cn} \Longrightarrow (\forall n\in \mathbb{N})\; R(A,F,n) \leq n|F|^{cn} \leq |F|^{(c+1)n}
\]
A profinite $G$ has UBERG if and only if there is a constant $c'$ such that $ R(A,F,n) \leq |F|^{c'n}$ for all $n$ and all finite fields $F$. 
\end{lem}
\begin{proof}
	The first assertion follows from the inequality $\sum_{j=1}^n |F|^{cj} \leq n|F|^{cn}$.	If $R(G,F,n) \leq |F|^{c'n}$ holds for all $n$ and all finite fields, then in particular $r(G,\F_p,n) \leq p^{c'n}$ for all prime numbers $p$ and so $G$ has UBERG.
	Conversely, if $G$ has UBERG,  then the inequality $r(G,F,n) \leq |F|^{c'n}$ follows from the proof of \cite[Lemma 6.8]{KV} using that every irreducible representation of degree $n$ is absolutely irreducible over a field extension of degree at most $n$.
\end{proof}

\begin{lem}\label{lem:basic-estimate}
Let $A$ be an $F$-algebra and let $B \subseteq A$ be a subalgebra. Suppose that $A$ can be generated by $d$ elements as a right $B$-module. Then
\begin{enumerate}[(i)]
\item\label{it:boundAB}
	$R(A,F,n) \leq d R(B,F,n)$
\item\label{it:boundBA}
	$R(B,F,n) \leq d R(A,F,dn)$
\end{enumerate}
\end{lem}
\begin{proof}
\eqref{it:boundAB}: For every finite dimensional simple $A$-module $V$, let $\psi(V)$ be some simple submodule of $V|_B$.
This defines a map $\psi \colon \bigcup_{i=1}^n\Irr(A,F,i) \to \bigcup_{i=1}^n\Irr(B,F,i)$.
We claim that the number of elements in each fibre is bounded by $d$ from above.
Consider the fibre over $\psi(V)$ and assume that $V$ has minimal $F$-dimension of all modules in the fibre.
Suppose that $\psi(V) \cong \psi(V')$, then
\[
	0 \neq \Hom_B(\psi(V), V'|_B) = \Hom_A(A \otimes_B \psi(V),V')
\]
and $V'$ is a simple factor of the induced module $A \otimes_B \psi(V)$. Since $\dim_F(V') \geq \dim_F(V)$ and $\dim_F(A \otimes_B \psi(V)) \leq d \dim_F(V)$, there are at most $d$ elements in each fibre. We deduce that
\[
	R(A,F,n) \leq d R(B,F,n).
\]

\eqref{it:boundBA}: For every finite dimensional simple $B$-module $W$
we choose a simple factor $\theta(W)$ of $A \otimes_B W$.
This defines a map $\theta\colon \bigcup_{i=1}^n\Irr(B,F,i) \to \bigcup_{j=1}^{dn}\Irr(A,F,j)$.
Again, we claim there are at most $d$ modules in each fibre.
Let $W$ have minimal dimension in the fibre over $\theta(W)$.
Suppose that $\theta(W') \cong \theta(W)$, then
\[
	0 \neq \Hom_A(A \otimes_B W', \theta(W)) = \Hom_B(W',\theta(W)|_B)
\]
and $W'$ is a simple submodule of $\theta(W)|_B$. Since $\dim(W) \leq \dim(W')$, there are at most $d$ such simple submodules, i.e.,
\[
	R(B,F,n) \leq d R(A,F,dn).\qedhere
\]
\end{proof}

\begin{cor}\label{cor:opensbgpreps}
Let $G$ be a profinite group and let $H$ be an open subgroup of index $h=[G:H]$.
Then $R(H,F,n) \leq hR(G,F,nh)$.
\end{cor}
\begin{proof}
Since $F\bra G \ket$ is $h$-generated as $F\bra H \ket$-module, Lemma \ref{lem:basic-estimate} implies
    $R(H,F,n) \leq h R(G,F,nh)$. 
\end{proof}

To proceed we will need the language of crossed representations and crossed projective representations. The notation is explained in Section \ref{sec:projective-crossed-etc}, details can be found in \cite{Karpilovsky}.

 \begin{cor}
	\label{cor:crossedreps}
	Let $E$ be a finite field and let $F$ be subfield with $|E:F|=e$. Let $\gamma$ be an action of $G$ on $E$ which fixes $F$. Then  $R(E^\gamma\bra G \ket,E,n) \leq eR(G,F,ne)$.
\end{cor}
\begin{proof}
Since $E^\gamma \bra G \ket$ can be generated by $e = [E:F]$ elements over $F\bra G\ket$, Lemma \ref{lem:basic-estimate} shows that
\[
	R(E^\gamma\bra G \ket,E,n) = R(E^\gamma\bra G \ket,F,ne) \leq e R(G,F,ne).\qedhere
\]
\end{proof}

\begin{prop}
	\label{projreps}
	Suppose $G$ is a profinite group, with an action $\gamma$ on a finite field $E$. Suppose $\alpha \in Z^2(G,E^\times)$ is a $2$-cocycle with respect to this action. Suppose the smallest $E$-dimension of an irreducible $E^\alpha \bra G \ket$-module is $\mu(\alpha)$. Then $R(E^\alpha\bra{G}\ket,E,n) \leq \mu(\alpha) R(E^\gamma\bra{G}\ket,E,n\mu(\alpha))$ for all $n$.
\end{prop}
\begin{proof}
Let $W$ be a simple $E^{\alpha^{-1}}\bra{G}\ket$-module of minimal dimension $\dim_E(W) = \mu(\alpha)$; the existence follows from the remark below Lemma \ref{lem:cocycles}.
We define a map $\rho\colon \Irr(E^\alpha\bra{G}\ket,E,n) \to \Irr(E^\gamma\bra{G}\ket,E,\mu(\alpha)n)$.
For a simple $E^\alpha\bra{G}\ket$-module $V$, we consider $V \otimes_E W$ and choose 
some simple quotient $\rho(V)$. Note that $\dim_E(\rho(V)) \leq \mu(\alpha)n$. 
We will bound the number of elements in each fibre. Suppose that $\rho(V) = \rho(V')$ and assume that $V'$ has minimal dimension in this fibre.
\[
	0 \neq \Hom_{E'\bra{G}\ket}(V \otimes_E W, \rho(V')) = \Hom_{E^\alpha\bra{G}\ket}(V, \rho(V')\otimes_E W^*) 
\]
i.e., $V$ is a simple submodule of $\rho(V')\otimes_E W^*$. The $\alpha$-representation $\rho(V')\otimes_E W^*$ has at most $\dim_E(W^*) = \mu(\alpha)$ many isomorphisms classes of simple submodules of dimension $\geq \dim_E(V')$.
Each fibre contains at most $\mu(\alpha)$ elements and so
\[
	R(E^\alpha\bra{G}\ket, E,n) 
	\leq \mu(\alpha)R(E^\gamma\bra{G}\ket,E,\mu(\alpha)n). \qedhere
\]
\end{proof}

For $K \unlhd G$, if we are given $\gamma: G/K \to \mathrm{Aut}(E)$ or $\alpha \in Z^2(G/K,E^\times)$ as above, we will also write $\gamma$ and $\alpha$ for the restrictions of $\gamma$ and $\alpha$ to $G$. The next theorem extends Clifford theory to twisted modules and it is our main tool to deal with extensions of UBERG groups.

\begin{thm}
	\label{modulestructure}
	Let $G$ be a profinite group, $K \unlhd G$, $W$ an irreducible $F \bra G \ket$-module. Write $V$ for an irreducible summand of $\Res^G_K W$, $E$ for the field $\End_{F \bra K \ket}(V)$, and $H$ for the inertial subgroup of $V$. Then there exists an action $\gamma$ of $H/K$ on $E$, a $2$-cocycle $\alpha \in Z^2(H/K,E^\times)$ with respect to this action, an extension $\ext(V)$ of $V$ to an $E^\alpha \bra H \ket$-module, and an irreducible $E^{\alpha^{-1}} \bra H/K \ket$-module $U'$ such that $$W \cong \Ind^G_H(U' \otimes_E \ext(V))$$ as $F \bra G \ket$-modules.
\end{thm}
\begin{proof}
	By \cite[Theorem 2.2.2(iii)]{Karpilovsky}, it is enough to prove the result for $H=G$, and we will assume for the rest of the proof that this is the case.
	
	Recall that the Schur index of $V$ over $F$ is $1$ since we are in positive characteristic (see \cite[Theorem 2.5.22]{Karpilovsky}), so we can apply \cite[Theorem 3.14.7]{Karpilovsky}. By \cite[Theorem 3.14.7(i)]{Karpilovsky}, $V$ extends to an $\alpha$-representation $\ext(V)$ of $G$ over $E$ with respect to some action $\gamma$ of $G$ on $E$, which is $F$-linear by \cite[Lemma 3.14.6(ii)]{Karpilovsky}; by \cite[Lemma 3.14.6(iii)]{Karpilovsky}, the action of $K$ on $E$ is trivial, so we may think of $\gamma$ as an action of $G/K$ on $E$. This gives the first part of the statement.
		
From now on we consider $W$ as $E^\gamma\bra G \ket$-module.
Let $\ext(V)^* = \Hom_E(\ext(V),E)$ be the dual of $\ext(V)$; this is an $E^{\alpha^{-1}}\bra G \ket$-module; see Lemma~\ref{lem:cocycles}.
Define
 \[ U' = (W \otimes_E \ext(V)^*)^K;\]
  the space of $K$-invariants in the $E^{\alpha^{-1}}\bra G \ket$-module $W \otimes_E \ext(V)^*$.
This is an $E^{\alpha^{-1}}\bra G/K \ket$-module.
Since $V$ is absolutely irreducible over $E$, we have $\dim_E (V\otimes_E V^*)^K = 1$. For some $m$, we have $\Res_K^G(W) \cong V^m$ and it follows that $U'$ has dimension $m$ over $E$. 

The trace $\ext(V)^* \otimes_E \ext(V) \to E$ induces a canonical homomorphism $W \otimes_E \ext(V)^* \otimes_E \ext(V) \to W$ of $E^\gamma\bra G \ket$-modules which restricts to a homomorphism $\phi \colon U' \otimes_E \ext(V) \to W$. It is easily checked that $\phi$ is surjective. We observe that 
\[
\dim_E(W) = m \dim_E(V) = \dim_E(U' ) \dim_E(V) = \dim_E (U' \otimes_E \ext(V))\]
and we deduce that $\phi$ is an isomorphism.
\end{proof}

\begin{thm}
	\label{UBERGbyUBERG}
	Suppose $G$ is a profinite group, $K \unlhd G$, $K$ and $G/K$ have UBERG, and the extension of $K$ by $G/K$ is split. Then $G$ has UBERG.
\end{thm}
\begin{proof}
	Fix $a$ such that $r(K,F,n) \leq |F|^{an}$ for all $n$ and $F$, and $b$ such that $r(G/K,F,n) \leq |F|^{bn}$ for all $n$ and $F$.
	
	We will count the irreducible $F\bra G \ket$-modules of dimension $n$. Suppose $W \in \Irr(G,F,n)$. Let $V$ be an irreducible summand of $\Res^G_K W$, of dimension $m$, say, and let $H$ be the inertial subgroup of $V$ with $|G:H|=h$. Write $E$ for the field $\End_{F \bra K \ket}(V)$ of degree $e \leq m$ over $F$.
	
	By Theorem \ref{modulestructure}, we can fix some extension $\ext(V)$ of $V$ to an $E^\alpha \bra G \ket$-module, some $2$-cocycle $\alpha \in Z^2(G,E^\times)$ associated to an action $\gamma$ of $G/K$ on $E$, and write $W$ in the form $\Ind^G_H(U' \otimes_E \ext(V))$, for some $U'$ irreducible $E^{\alpha^{-1}}\bra{H/K}\ket$-module of dimension $n/hm$ over $E$.
	
	Let $\mu(\alpha^{-1})$ be the minimal $E$-dimension of an irreducible $E^{\alpha^{-1}}\bra{H/K}\ket$-module. Taking duals gives $\mu(\alpha) = \mu(\alpha^{-1})$; see Lemma \ref{lem:cocycles}. We claim that $\mu(\alpha) \leq \frac{m}{e}$. Since the extension $K \to G \to G/K$ is split, $K \to H \to H/K$ is too, so there is a complement $H' \cong H/K$ of $K$ in $H$. Restricting $\ext(V)$ to $H'$, we obtain an $E^{\alpha}\bra{H/K}\ket$-module of dimension $m/e$. Any irreducible factor is an irreducible $E^{\alpha}\bra{H/K}\ket$-module of $E$-dimension at most $m/e$.

    By Lemma \ref{rvsR}, Corollary \ref{cor:opensbgpreps}, Corollary \ref{cor:crossedreps} and Proposition \ref{projreps} we have
	\begin{align*}
	 	r(E^{\alpha^{-1}}\bra{H/K}\ket,E,n/mh)  &\leq R(E^{\alpha^{-1}}\bra{H/K}\ket,E,n/mh)\\
	 	&\leq \mu(\alpha)R(E^{\gamma}\bra{H/K}\ket,E,n\mu(\alpha)/mh)\\
		&\leq \mu(\alpha)eR(H/K,F,n\mu(\alpha)e/mh)\\
		&\leq \mu(\alpha)eh R(G/K,F,n\mu(\alpha)e/m)\\
		&\leq mh R(G/K,F,n)\\
		&\leq n^2 |F|^{bn},
	\end{align*}
	where we use the inequalities $\mu(\alpha) \leq \frac{m}{e}$ and $mh \leq n$ in the last two steps.
	So there are at most $r(K,F,m) \leq |F|^{am}$ choices for $V$, and at most $n^2|F|^{bn}$ choices for $U'$. Hence the number of possible $W \in \Irr(G,F,n)$ whose restriction to $K$ has irreducible components of dimension $m$ is at most $n^2|F|^{am}|F|^{bn}$, and therefore 
	\[
		r(G,F,n) \leq \sum_{m=1}^n n^2|F|^{am+bn} \leq n^3|F|^{(a+b)n} \leq |F|^{(a+b+3)n}\]
	for all $n$ and $F$, as required.
\end{proof}

\begin{cor}
	\label{UBERGbyUBERG2}
	Suppose $G$ is a profinite group, $K \unlhd G$ has UBERG, and the universal Frattini cover $Q$ of $G/K$ has UBERG. Then $G$ has UBERG.
\end{cor}
\begin{proof}
	Consider the diagram
	\[\xymatrix{
		K \ar[r] \ar@{=}[d] & L \ar[r] \ar[d] & Q \ar@{->>}[d] \\
		K \ar[r] & G \ar[r] & G/K
	}\]
	in which the right-hand square is a pull-back, and the rows are short exact. The map $L \to G$ is epic, so it is enough to show $L$ has UBERG. Since $Q$ is projective, the top row is split, and the result follows from Theorem \ref{UBERGbyUBERG}.
\end{proof}

\begin{cor}
	\label{UBERGbyPUBERG}
	Suppose $G$ is a profinite group, $K \unlhd G$ has UBERG, and $G/K$ is finitely generated and has proj-UBERG (in particular, this holds if $G/K$ is PFG by Corollary \ref{projUBERG}). Then $G$ has UBERG.
\end{cor}
\begin{proof}
	This follows from Corollary \ref{UBERGbyUBERG2} by Theorem \ref{cover/quotient}.
\end{proof}

Since PFR profinite groups are precisely the finitely presented profinite groups with UBERG, and positively finitely presented profinite groups are precisely the finitely presented PFG profinite groups, we get:

\begin{cor}
	PFR-by-positively finite presented profinite groups are PFR.
\end{cor}

Using the same arguments, we can show similarly:

\begin{thm}
	\label{PFPnbyUBERG}
	Let $G$ be a profinite group, $K \unlhd G$. Suppose $G/K$ has UBERG, the extension splits, and $K$ has type $\PFP_n$ over a commutative profinite ring $R$. Then $K$ has relative type $\PFP_n$ in $G$ over $R$.
\end{thm}
\begin{proof}
	We can assume $R$ is PFG as an $R$-module; if not, $K$ does not have type $\PFP_0$ over $R$, so the result is vacuously true. Summands of modules of type $\PFP_n$ over any ring have type $\PFP_n$ by \cite[Theorem 4.9]{CCV}, since the $\mathrm{Ext}$-functors are additive, so it suffices to show $R \bra G/K \ket$ has type $\PFP_n$ over $R \bra G \ket$.
	
	Start with a type $\PFP_n$ resolution for $R$ as an $R \bra K \ket$-module, and apply $\Ind^G_K$: we just need to show $\Ind^G_K P$ is PFG for $P$ a PFG projective $R \bra K \ket$-module.
	
	For an irreducible $R \bra G \ket$-module $W$, from the definition of $\Ind^G_K$ we have $$\Hom_{R \bra G \ket}(\Ind^G_K P,W) \cong \Hom_{R \bra K \ket}(P,\Res^G_K W),$$ so by \cite[Theorem 4.10]{CCV} it is enough to polynomially bound the number of (isomorphism types of) irreducible $R \bra G \ket$-modules $W$ such that $\Hom_{R \bra K \ket}(P,\Res^G_K W)$ is non-trivial. Clifford theory tells us $\Res^G_K W$ is a direct sum of irreducible $R \bra K \ket$-modules, so a non-trivial map $P \to \Res^G_K W$ shows that some irreducible summand $V$ appears as a quotient of $P$. Since $P$ is PFG, the number of possibilities for $V$ is polynomially bounded in the order of $V$. So we fix $V$, and consider only the irreducible $R \bra G \ket$-modules $W$ such that $\Res^G_K W$ contains $V$ as a summand. Since $G/K$ has UBERG and the extension is split, the proof of Theorem \ref{UBERGbyUBERG} shows that the number of possibilities for such $W$ is polynomially bounded in the order of $W$ (uniformly over the possible $V$), and the conclusion follows.
\end{proof}

\begin{cor}
	\label{PFPnbyPFPm}
	In the situation of the theorem above:
	\begin{enumerate}[(i)]
		\item if $G/K$ has type $\PFP_m$, $G$ has type $\PFP_{\min(m,n)}$;
		\item if $G$ has type $\PFP_m$, $G/K$ has type $\PFP_{\min(m,n+1)}$.
	\end{enumerate}
\end{cor}
\begin{proof}
	This follows immediately by \cite[Theorem 5.22]{CCV}.
\end{proof}

As for Corollary \ref{UBERGbyPUBERG}, we get results analogous to Theorem \ref{PFPnbyUBERG} and Corollary \ref{PFPnbyPFPm} when we consider $G$ with $K \unlhd G$ of type $\PFP_n$, such that $G/K$ is finitely generated and has proj-UBERG.

\section{Equivalent conditions for \texorpdfstring{$\FP_1$}{FP1} and \texorpdfstring{$\PFP_1$}{PFP1}}\label{sec:PFP1}
In this section we give a semi-structural condition which is necessary and sufficient for a finitely generated profinite group $G$ to have type $\PFP_1$. As a preparation we briefly discuss an equivalent condition for type $\FP_1$.

\subsection{Type \texorpdfstring{$\FP_1$}{FP1}}

A necessary and sufficient condition for a profinite group to have type $\FP_1$ is given in \cite[Corollary 5.10]{CCV}.

\begin{prop}
	\label{prop:augmentation}
	$G$ has type $\FP_1$ if and only if there exists $d \in \mathbb{N}$ such that $(\delta_G(M)+h'_G(M))/r_G(M) \leq d$ for any $M \in \mathrm{Irr}(\hat{\Z}\bra G \ket)$. For $f: P \to \hat{\Z}$ a projective cover in the category of $\hat{\Z}\bra G \ket$-modules, the minimum number of generators of $\ker(f)$ is $$\sup_{M \in \mathrm{Irr}(\hat{\Z}\bra G \ket)} \left \lceil \frac{\delta_G(M)+h'_G(M)}{r_G(M)} \right \rceil.$$
\end{prop}

Here $\mathrm{Irr}(\hat{\Z}\bra G \ket)$ is the set of irreducible $\hat{\Z}\bra G \ket$-modules, $\delta_G(M)$ is the number of non-Frattini chief factors $G$-isomorphic to $M$ in a chief series of $G$, $r_G(M)$ is defined by $M \cong \End_G(M)^{r_G(M)}$ as $\End_G(M)$-modules, and $h'(M)$ to be the dimension of $H^1(G/C_G(M),M)$ over $\End_G(M)$.

We can reformulate Proposition~\ref{prop:augmentation} in terms of the crown-based powers appearing as quotients of $G$. 
\begin{cor}
	\label{crownFP1}
	A profinite group $G$ has type $\FP_1$ if and only if there exists $d$ such that, for any 
	monolithic primitive group $L$ with abelian minimal normal subgroup $M$, the size $k$ of a crown-based power $L_k$ of $L$ which occurs as a quotient of $G$ is at most $dr_G(M)$.
\end{cor}
\begin{proof}
Let $M \in \mathrm{Irr}(\hat{\Z}\bra G \ket)$.
Note that $h'_G(M)/r_G(M) < 1$ by \cite[Theorem A]{AG}.
So, by Proposition \ref{prop:augmentation}, $G$ has type $\FP_1$ exactly if $\delta_G(M)/r_G(M) \leq d$ for some $d$ which does not depend on $M$.

If $L_{k}$ is a crown-based power of a monolithic primitive group $L$ with abelian minimal normal subgroup $M$ which appears as a quotient of $G$, then, by \cite[Theorem 11]{DL}, $k$ is at most the cardinality of the set of chief factors of $G$ which are $G$-equivalent to $M$. For abelian chief factors, $G$-equivalence is the same as $G$-isomorphism by the remark after \cite[Definition 1]{DL}, i.e., $k \leq \delta_G(M)$. On the other hand, for every $k'\leq \delta_G(M)$ the crown-based power $L_{k'}$ appears as a quotient of $G$ by \cite[Theorem 11]{DL}. 
\end{proof}

It is interesting to compare this result with the equivalent condition for PFG given in \cite[Theorem 11.1(3)]{JP}, which is expressed entirely in terms of sizes of the crown-based powers of monolithic groups with non-abelian minimal normal subgroup which appear as quotients of $G$. On the other hand, the minimum number of generators for a profinite group can be determined by the crown-based powers of all monolithic primitive groups which appear as quotients of $G$ (see \cite[Lemma 4.2]{Damian}).

We will see below, in Remark \ref{PFP1rem}(ii), that whether a finitely generated profinite group with at most exponential subgroup growth has type $\PFP_1$ can also be determined by the crown-based powers of all monolithic primitive groups which appear as quotients of $G$.

\subsection{Type \texorpdfstring{$\PFP_1$}{PFP1}}

Recall from \cite{CCV} that a finitely generated profinite group $G$ has type $\PFP_1$ if and only if there is some constant $c$ such that, for all $m$, the number of irreducible $G$-modules $M$ of order $m$ such that $H^1(G,M) \neq 0$ is at most $m^c$. To control the number of such $G$-modules, we use the following result.

\begin{prop}{\cite[(2.10)]{AG}}
	\label{AG}
	$|H^1(G,M)| = q^n|H^1(G/C_G(M),M)|$, where $q = |\End_G(M)|$ and $n$ is the number of non-Frattini chief factors of $G$ $G$-isomorphic to $M$.
\end{prop}

Thus we may consider separately the irreducible $G$-modules $M$ such that $H^1(G/C_G(M),M) \neq 0$, and the $M$ such that $n \neq 0$. (Note these sets need not be disjoint.)

\begin{define}\label{defn:A}
	We say a profinite group $G$ satisfies \emph{condition (A)} if there is some constant $a$ such that, for all $m$, the number of $G$-isomorphism classes of non-Frattini abelian chief factors $M$ of $G$ of order $m$ is at most $m^a$.
\end{define}

\begin{lem}
	\label{(A)}
	If $G$ has type $\FP_1$ and satisfies (A), there is a constant $a'$ such that, for all $m$, the number of non-Frattini abelian chief factors $M$ of $G$ of order $m$ is at most $m^{a'}$.
\end{lem}
\begin{proof}
	Since $G$ has type $\FP_1$, there is some $a''$ such that $|H^1(G,M)| \leq |M|^{a''}$. So by Proposition \ref{AG}, there are at most $m^{a''}$ non-Frattini chief factors of $G$ which are $G$-isomorphic to $M$, and hence at most $m^{a+a''}$ non-Frattini abelian chief factors of $G$ of order $m$.
\end{proof}

We progress by recalling the following result.

\begin{prop}{\cite[Lemma 5.2]{GKKL}}
	\label{GKKL}
	Suppose $T$ is a finite group, and $M$ is a faithful $T$-module such that $H^1(T,M) \neq 0$. Then $T$ has a non-abelian unique minimal normal subgroup.
\end{prop}

Note that, for such a $T$ with non-abelian unique minimal normal subgroup $K$, $C_T(K)$ is a normal subgroup of $T$ so must be trivial. So $T$ is the monolithic group associated with the non-abelian characteristically simple group $K$. For the rest of the section, $T$ will denote such a monolithic group, and $K$ its non-abelian minimal normal subgroup.

We define \emph{the $H^1$-length of $T$}, $l^{H^1}(T)$, to be the order of the smallest faithful irreducible $T$-module $M$ (over any field) such that $H^1(T,M) \neq 0$. If no such $M$ exists, we set $l^{H^1}(T) = \infty$.

Given a profinite group $G$, we define \emph{the $T$-rank of $G$}, $\rk_T(G)$, to be the maximal $r \geq 0$ such that there is an epimorphism $\phi$ from $G$ to a subdirect product of $T^r$ such that $K^r \leq \phi(G)$.

\begin{lem}
	$\rk_T(G)$ is the number of non-abelian chief factors $A$ of $G$ such that $G/C_G(A) \cong T$ (and thus $A \cong K$).
\end{lem}
\begin{proof}
	If there are $s$ such chief factors $A_1, \ldots, A_s$, pick a map $G \to G/C_G(A_i)$ for each $i$. Then the product $\phi: G \to \prod_i G/C_G(A_i)$ of these maps has as its image a subdirect product, with $K^s \leq \phi(G)$. So $s \leq \rk_T(G)$.
	
	Conversely, given $\phi: G \to T^r$ with $r=\rk_T(G)$ such that $K^r \leq \phi(G)$, the image $\phi(G)$ has $r$ composition factors $K_i$ such that $\phi(G)/C_{\phi(G)}(K_i) \cong T$: these are the $r$ copies of $K$. So $G$ has at least this many, and $s \geq \rk_T(G)$.
\end{proof}

\begin{define}\label{defn:B_C}
	\begin{enumerate}[(i)]
		\item We say that $G$ satisfies \emph{condition (B)} if there is some $b$ such that
		 for all $m$ and all vector spaces $M$ with $|M|=m$ the number of $GL(M)$-conjugacy classes of irreducible subgroups $L$ of $GL(M)$, with $H^1(L,M) \neq 0$, appearing as quotients of $G$ is at most $m^b$. Note, by Proposition \ref{GKKL}, that in this case $L$ must be a monolithic group with a non-abelian minimal normal subgroup.
		\item We say that $G$ satisfies \emph{condition (C)} if, for all monolithic groups $L$ with non-abelian minimal normal subgroup, there is some $c$ such that $\rk_L(G) \leq l^{H^1}(L)^c$.
	\end{enumerate}
\end{define}

Write $Epi(G,T)_T$ for the set of $T$-conjugacy classes of epimorphisms $G \to T$. Then, as for Lemma \ref{epi}, we get that $Epi(G,T)_T = |Epi(G,T)|/|T|$, because $T$ has trivial centre, so acts faithfully on $Epi(G,T)$ by conjugation; we will use this repeatedly in the proof below.

\begin{thm}
	\label{thm:PFP1}
	Assume $G$ has type $\FP_1$. The following are equivalent:
	\begin{enumerate}[(i)]
		\item $G$ has type $\PFP_1$.
		\item There is some $c$ such that the number of simple $G$-modules $M$ of order $n$ such that $H^1(G,M) \neq 0$ is $\leq n^c$.
		\item $G$ satisfies (A) and (B), and there is some $c$ such that for any group $L$ associated with a characteristically simple non-abelian group $A$, $|Epi(G,L)| \leq |L|l^{H^1}(L)^c$.
		\item $G$ satisfies (A), (B) and (C).
	\end{enumerate}
\end{thm}
\begin{proof}
	(i) $\Leftrightarrow$ (ii) is in \cite[Corollary 5.13]{CCV}.
	
	(ii) $\Rightarrow$ (iii): For any $L$, there is some simple $L$-module $M$ such that $|M| = l^{H^1}(L)$ and $H^1(L,M) \neq 0$. Now, for each conjugacy class of epimorphisms $G \to L$, restricting the $L$-action to $G$ gives a $G$-module $M$ such that $H^1(G,M) \neq 0$ by Proposition \ref{AG}. So (ii) implies $|Epi(G,L)_L| \leq l^{H^1}(L)^c$. $G$ satisfies (A) by Proposition \ref{AG}; (B) is clear from the definition.
	(iii) $\Rightarrow$ (ii): For any $L$, $|Epi(G,L)_L| \leq l^{H^1}(L)^c$. Since $G$ satisfies (A), there is some $a$ such that the number of simple $G$-modules $M$ of order $n$ such that $H^1(G/C_G(M),M)=0$ and $H^1(G,M) \neq 0$ is $\leq n^a$. If $M$ is a simple $G$-module of order $n$ such that $H^1(G/C_G(M),M) \neq 0$, by (B), there are $\leq n^b$ possibilities for $L = G/C_G(M)$ up to conjugacy in $GL(M)$. Fix one such possibility. By Proposition \ref{GKKL}, $L=G/C_G(M)$ is the monolithic group associated to its non-abelian minimal normal subgroup $A$; clearly $l^{H^1}(L) \leq n$, so $|Epi(G,L)_L| \leq n^c$. We conclude that condition (ii) holds with constant $a+b+c$.
	
	(iii) $\Rightarrow$ (iv): If for all $c$ there is some $L$ with $\rk_L(G) > l^{H^1}(L)^c$, we have an epimorphism $\phi$ from $G$ to a subdirect product of $L^{\rk_L(G)}$ such that $A^{\rk_L(G)} \leq \phi(G)$. The projections from $G$ onto each factor of the subdirect product have different kernels, so we know $|\Epi(G,L)_L| > l^{H^1}(L)^c$; therefore $|\Epi(G,L)| > |L| l^{H^1}(L)^c$.
	
	(iv) $\Rightarrow$ (iii): By \cite[Lemma 2.12, Remark 2.14]{JP}, we know $|\Epi(G,L)| \leq \rk_L(G)(5|Out(S)|)^s|L|$, where $A \cong S^s$ with $S$ simple. We have $|Out(S)| \leq l(S)$ by \cite[Lemma 7.7]{GMP}, and clearly $l(S) \geq 2$ for all $S$, so $$(5|Out(S)|)^s \leq (5l(S))^s \leq l(S)^{4s} \leq l^{proj}(A)^4 \leq l^{H^1}(L)^4;$$ hence $|\Epi(G,L)| \leq |L| l^{H^1}(L)^{c+4}$, as required.
\end{proof}
\begin{rem}
	\label{PFP1rem}
	\begin{enumerate}[(i)]
		\item In the proof of \cite[Proposition 7.1]{JP}, to which (iv) $\Rightarrow$ (iii) in our theorem is analogous, the inequality $5|Out(S)| \leq l^{proj}(S)$, for all non-abelian simple $S$, is implicitly used. In fact this is not true, e.g. $S = PSL_3(\F_2)$ has $l^{proj}(S)=8$ and $|Out(S)|=2$. But it is true `up to a constant', and this makes no difference to the argument, as here.
		\item We would like to have an equivalent condition for $\PFP_1$ using crown-based powers, along the lines of Theorem \ref{crownpowersUBERG} or \cite[Theorem 11.1(3)]{JP}. Those conditions work because PFG and finitely generated UBERG groups have at most exponential subgroup growth (see \cite[Theorem 10.2]{JP}, \cite[Proposition 5.4]{KV}). This fails in general for finitely generated groups of type $\PFP_1$: see Corollary \ref{superexponential} below. If we restricted to finitely generated groups with at most exponential subgroup growth, an equivalent condition in terms of crown-based powers would hold.
		\item Note, from the proof of the theorem, that the only place we require $G$ to have type $\FP_1$ is in (ii) $\Rightarrow$ (i). With no assumptions on $G$, (i) $\Rightarrow$ (ii) $\Leftrightarrow$ (iii) $\Leftrightarrow$ (iv).
	\end{enumerate}
\end{rem}
\begin{cor}
\label{cor:PFP1-products}
Let $(Q_n)_n$ be a sequence of finite quasisimple groups and let $(a(n))_n$ be a sequence of natural numbers. Let $Z_n$ be the center of $Q_n$ and assume that all $Q_n/Z_n$ are pairwise distinct.
Then $G = \prod_{n}Q_n^{a(n)}$ has type $\PFP_1$ if and only if there is $c > 0$ such that
$a(n) \leq l^{H^1}(Q_n/Z_n)^c$ for all $n$.
\end{cor}
\begin{proof}
Clearly, $\Phi(G) = \prod_n Z_n^{a(n)}$ and $G$ does not have abelian non-Frattini chief factors.
Using Proposition \ref{prop:augmentation} and $h'_G(M)/r_G(M) \leq 1$ (see \cite[Theorem A]{AG}) it follows that $G$ always has type $\FP_1$ and satisfies condition (A).

We verify that condition (B) is satisfied. By Proposition \ref{GKKL}, for any vector space $M$, $|M|=m$, any irreducible subgroup $L$ of $GL(M)$ with $H^1(L,M) \neq 0$ appearing as a quotient of $G$ must be isomorphic to some $Q_n/Z_n$ and is thus $2$-generated. By \cite[Proposition 6.1]{JP}, there is some $b$ such that, for all $M$, the number of $2$-generated irreducible subgroups of $GL(M)$ up to conjugacy is $\leq m^b$, so $G$ satisfies (B). 

The monolithic primitive factors of $G$ are exactly the groups $T_n = Q_n/Z_n$ and we have $\mathrm{rk}_{T_n}(G) = a(n)$. So condition (C) holds exactly if there is $c > 0$ such that $a(n) \leq l^{H^1}(T_n)^c$ for all $n$.
\end{proof}
\begin{cor}
	\label{fgPFP1}
	Suppose $G$ is finitely generated. The following are equivalent:
	\begin{enumerate}[(i)]
		\item $G$ has type $\PFP_1$.
		\item $G$ satisfies (A) and (C).
	\end{enumerate}
\end{cor}
\begin{proof}
	For $G$ finitely generated, (B) holds by \cite[Proposition 6.1]{JP}.
\end{proof}

We can use Theorem \ref{thm:PFP1} to re-prove the following result, which was proved in \cite{CCV} in a different way.

\begin{cor}
	Suppose $G$ is has UBERG and type $\FP_1$. Then $G$ has type $\PFP_1$.
\end{cor}
\begin{proof}
	We show that $G$ satisfies (A), (B) and (C). Since $G$ has UBERG, there is some $a$ such that the number of irreducible $G$-modules of order $m$ is at most $m^a$. So (A) and (B) are immediate. By Theorem \ref{necessary}, there is some $b$ such that, for all monolithic groups $L$ with a non-abelian minimal normal subgroup, the size of a crown-based power of $L$ occurring as a quotient of $G$ is $\leq l^{lin}(K)^b$. As in Theorem \ref{crownpowersUBERG}, we deduce from Lemma~\ref{lem:aut-1} that $\rk_L(G) \leq l^{lin}(K)^c$ for some constant $c$. Finally, by Proposition \ref{primitiverep}, $l^{H^1}(T) \geq l^{lin}(K)^e$ for the constant $e$ used in that proposition, so we conclude $G$ satisfies (C).
\end{proof}

In fact, in Theorem~\ref{thm:UBERG+FP_1=FG} we will prove a stronger result: a profinite group with UBERG and $\FP_1$ must be finitely generated. Now, we will deal with extensions of groups of type $\PFP_1$. 

\begin{thm}\label{thm:PFP1_by_PFP1}
Let $G$ be a profinite group and $N \unlhd G$. If $N$ and $G/N$ have type $\PFP_1$, then $G$ has type $\PFP_1$
\end{thm}
\begin{proof}
Since $\FP_1$ is preserved under extensions, we know that $G$ has type $\FP_1$.
Let $r$ be such that $|H^1(N,U)| \leq |U|^{r}$ for all finite $N$-modules $U$.
Since $N$ is of type $\PFP_1$ there is $d \geq 0$ be such the number of simple $N$-modules of order $\leq m$ is at most $m^d$ (and similarly there is a constant $d'$ for $G/N$).
In view of (ii) in Theorem \ref{thm:PFP1} we need to count the finite simple $G$-modules $V$ with $H^1(G,V) \neq 0$ and $|V| = m$.
We will distinguish whether or not $N$ acts trivially on $V$.

\textit{Case 1}: $\Res_N^G(V)$ is non-trivial. In this case the restriction is semisimple and contains a non-trivial direct summand $U$.
Then
\[
	\Hom_N(\Res_N^G(V), U) = \Hom_G(V,\mathrm{Coind}_N^G(U))
\]
and $V$ is a simple submodule of the coinduced module $\mathrm{Coind}_N^G(U)$.
We fix $U$ and define $t(G,U)$ to be the number of simple $G$-modules $V$ which contain $U$ and satisfy $H^1(G,V) \neq 0$. 
Define $S$ to be the sum of all simple submodules of $\mathrm{Coind}_N^G(U)$.
Then $t(G,U) \leq |H^1(G,S)|$.
The short exact sequence $S \to \mathrm{Coind}_N^G(U) \to  \mathrm{Coind}_N^G(U)/S$ gives rise to a long exact sequence
\[
	H^0(G,\mathrm{Coind}_N^G(U)/S) \longrightarrow H^1(G,S) \longrightarrow  H^1(G,\mathrm{Coind}_N^G(U))
\]
where the right hand side is isomorphic to $H^1(N,U)$ by Shapiro's Lemma \cite[6.10.5]{RZ}.
We note that composition factors of $\Res_N^G(\mathrm{Coind}_N^G(U))$ are $G$-conjugates of $U$ and, in particular, are non-trivial.
This means that all composition factors of $\mathrm{Coind}_N^G(U)/S$ are non-trivial and  $H^0(G,\mathrm{Coind}_N^G(U)/S) = 0$.
We deduce $t(G,U) \leq |H^1(G,S)| \leq |H^1(N,U)|$
and so if $t(G,U) \neq 0$ then $H^1(N,U) \neq 0$ and $|H^1(N,U)| \leq |U|^r \leq m^r$. 
Since $N$ has $\PFP_1$ there are at most $m^d$ distinct modules $U$ with $|U| \leq m$ and $H^1(N,U) \neq 0$.

\textit{Case 2}: The action on $V$ factors through $G/N$. Say $V$ an $\F_p\bra{G}\ket$-module.
Consider the initial piece of the five term exact sequence of the Lyndon-Hochschild-Serre spectral sequence (e.g. \cite[Theorem 3.7]{Ged}))
\[
	0 \longrightarrow H^1(G/N,V) \longrightarrow H^1(G,V) \longrightarrow H^1(N,V)^{G/N} 
\]
and observe that the last term is 
\[
	H^1(N,V)^{G/N} = \Hom_{G/N}(N,V) = \Hom_{G/N}(H_1(N,\F_p),V).
\]
So if $H^1(G,V) \neq 0$, then either $H^1(G/N,V) \neq 0$ or $V$ is a $G/N$-factor of $H_1(N,\F_p)$. Since $G/N$ has $\PFP_1$, there are at most $m^{d'}$ simple $G/N$-modules of order $m$ of the former kind. Moreover, since 
$|H_1(N,\F_p)| \leq p^r$ is follows that $H_1(N,\F_p)$ has at most $r$ distinct simple factors; this means, there are at most $r$ modules of the latter kind. 

In total there are at most 
\[
	m^{r+d} + rm^{d'}\leq m^{r(d+d')}
\]
simple modules $V$ of order $|V| = m$ and $H^1(G,V) \neq 0$, i.e., $G$ has type $\PFP_1$.
\end{proof}

\section{\texorpdfstring{$\FP_1$}{FP1}, UBERG and pronilpotent groups}\label{sec:pronilp}

Type $\FP_n$ and $\PFP_n$ will always be understood to mean over $\hat{\mathbb{Z}}$ in this paper, except when we specify otherwise. In this section we take a closer look at these properties for groups with normal pronilpotent subgroups.
\begin{thm}\label{thm:fp1-pronilpotent}
	Let $\widetilde{G}$ be a profinite group with a normal pronilpotent subgroup $P \trianglelefteq \widetilde{G}$ such that $G = \widetilde{G}/P$ is finitely generated. If $\widetilde{G}$ is $\FP_1$, then $\widetilde{G}$ is finitely generated.
\end{thm}
The following proposition covers the special case where $P = A$ is abelian and $\widetilde{G}$ is a split extension. The general result will be reduced to this situation.
\begin{prop}\label{prop:fp1-pronilpotent_split}
	Let $G$ be a finitely generated profinite group and let $A$ be an abelian profinite $G$-module. If $\widetilde{G} = A \rtimes G$ is $\FP_1$, then $\widetilde{G}$ is finitely generated.
\end{prop}
\begin{proof}
	For every $G$-module $M$ the sequence
	\[
	0 \longrightarrow H^1(G,M) \longrightarrow H^1(\widetilde{G},M) \longrightarrow \Hom_G(A,M) \longrightarrow 0
	\]
	is exact. This follows from looking at the 5-term exact sequence and the observation, that the map $H^2(G,M) \to H^2(\widetilde{G},M)$ admits a splitting.
	In particular
	\[
	|H^1(\widetilde{G},M)| = |H^1(G,M)|\cdot|\Hom_G(A,M)|.
	\]
	The head $A^h$ of $A$ is an infinite product of simple $G$-module
	\[
	A^h \cong \prod_{M \text{simple}} M^{m(M)}
	\]
	where each finite simple module $M$ occurs a certain number of times. For every finite simple $G$-module $M$ we have
	\[
	|H^1(\widetilde{G},M)| = |H^1(G,M)|\cdot|\mathrm{End}_G(M)|^{m(M)}.
	\]
	By assumption $\widetilde{G}$ has $\FP_1$ and finiteness of $H^1(\widetilde{G},M)$ implies that $m(M)$ is finite for every simple module $M$. If $M$ is a simple $\bbF_p\bra{G}\ket$ module, then $\End_G(M) = \bbF_{q_M}$ for a prime power $q_M = p^{f_M}$ (with $f_M \in \mathbb{N}$
	and $|M| = q^{k_M} = p^{f_Mk_M}$ with $k_M \in \mathbb{N}$.
	We observe that $M$ occurs exactly $k_M$ times in the head of $\mathbb{Z}\bra{G}\ket$. In particular, a module of the form $M^m$ is generated by no less than $\lfloor\frac{m}{k_M}\rfloor$ elements.
	
	We claim that $\frac{m(M)}{k_M}$ is bounded independently of $M$, so that $A$ is a finitely generated $G$-module and is a finitely normally generated subgroup of $\widetilde{G}$. In particular, $\widetilde{G}$ is finitely generated.
	
	Since $\widetilde{G}$ is of type $FP_1$, there is a constant $b$ such that for all primes $p$ and all simple modules $M$ of order $|M|=p^c$ we have
	\begin{align*}
		|M|^b \geq |H^1(\widetilde{G},M)|
		\geq |\End_G(M)|^{m(M)}	
	\end{align*}
	Let $M$ be a simple module and write $|M|=p^{k_{M}f_{M}}$ ,then
	\[
	p^{f_Mk_Mb} = |M|^b \geq |\End_G(M)|^{m(M)} = p^{f_{M}m(M)}
	\]
	and we deduce $bf_{M}k_{M} \geq f_{M}m(M)$ and hence
	$b \geq \frac{m(M)}{k_{M}}$.
\end{proof}
\begin{proof}[Proof of Theorem \ref{thm:fp1-pronilpotent}]
	Let $\Phi(\widetilde{G})$ be the Frattini subgroup of $\widetilde{G}$. Then $\widetilde{G}$ is finitely generated if and only if $\widetilde{G}/(P\cap\Phi(\widetilde{G}))$ is finitely generated. 
	As factor group of an $FP_1$ group $\widetilde{G}/(P\cap\Phi(\widetilde{G}))$ is still $FP_1$. In particular, we may assume that $\Phi(\widetilde{G}) \cap P = \{e\}$. In this case we also have $\Phi(P) \subseteq P \cap \Phi(\widetilde{G}) = \{e\}$, i.e., $P$ is a product of profinite abelian groups of prime exponent.
	
	We claim that $\widetilde{G}$ is a split extension by $P$. As a first step we find a minimal supplement.
	
	Let $\mathcal{S}$ be the set of closed supplements to $P$ in $\widetilde{G}$, i.e., the set of closed subgroups  $K \leq_c \widetilde{G}$ with $PK = \widetilde{G}$.
	Every descending chain $C$ in $\mathcal{S}$ satisfies $P\bigcap_{K \in C} K = \widetilde{G}$. Indeed, this follows from compactness: for all $g\in \widetilde{G}$, $K \in C$ the set $K_g = \{k \in K \mid g \in Pk \}$ is compact and hence $\bigcap_{K\in C} K_g \neq \emptyset$.
	By Zorn's Lemma there is a minimal supplement $K \leq_c \widetilde{G}$ to $P$. 
	
	We show that a minimal supplement is a complement; i.e. $K \cap P = \{e\}$. Suppose for a contradiction that $K \cap P$ is non-trivial. Since $\Phi(\widetilde{G})\cap P = \{e\}$, there is a maximal subgroup $H \leq \widetilde{G}$ which does not contain $K \cap P$. 
	Note that $K \cap P$ is normal in $\widetilde{G}$, since $K$ is a supplement and $P$ is abelian. Therefore $H(P\cap K) = \widetilde{G}$. Let $k \in K$. Then $k = h q$ with $h \in H$ and $q \in K\cap P$. Hence $h \in K \cap H$. We deduce that $K \subseteq (H\cap K)(K \cap P)$ and so $(H\cap K)P = KP = \widetilde{G}$. This means that $H\cap K$ is a supplement to $P$ and provides a contradiction to minimality of $K$. 
	
	As $\tilde{G}$ is a split extension by $P$, Proposition~\ref{prop:fp1-pronilpotent_split} concludes the proof.
\end{proof}
As an application we deduce that the properties of interest to us are all equivalent for pronilpotent groups.
\begin{thm}\label{thm:pronilpotent}
	Let $P$ be pronilpotent group. The following are equivalent:
	\begin{enumerate}[(i)]
		\item\label{it:p-fg} $P$ is finitely generated,
		\item\label{it:p-uberg} $P$ has UBERG,
		\item\label{it:p-pfp1} $P$ is of type $\PFP_1$.
		\item\label{it:p-fp1} $P$ is of type $\FP_1$,
	\end{enumerate}
\end{thm}
\begin{proof}
	The implications
	``\eqref{it:p-pfp1} $\Rightarrow$ \eqref{it:p-fp1}'' and  ``\eqref{it:p-fg} $\Rightarrow$ \eqref{it:p-fp1}'' are clear. If $P$ has $\FP_1$, then Theorem \ref{thm:fp1-pronilpotent} implies that $P$ is finitely generated (since the trivial group is finitely generated).
	If $P$ is finitely generated and has UBERG, then it has type $\PFP_1$ (see \cite{CCV}). A finitely generated pronilpotent group $P$ has UBERG by \cite[Corollary 6.12]{KV}.
	
	It remains to show that that \eqref{it:p-uberg} implies \eqref{it:p-fg}.
	We argue by contraposition and assume that $P$ is cannot be finitely generated.
	Consider the Frattini quotient $P/\Phi(P)$ of $P$: this is a product of the form $\prod_p (\mathbb{Z}/p\mathbb{Z})^{d_p(G)}$, where $d_p(G)$ is the minimal number of generators of the Sylow $p$-subgroup of $G$. As a function of the primes $p$, $d_p(G)$ is unbounded. We will show this quotient does not have UBERG. We will show the number of irreducible representations order at most $k$ grows faster than polynomially in $k$.
	
	For each $p$, write $M_p$ for a non-trivial irreducible $\mathbb{Z}/p\mathbb{Z}$-module of minimal order. The number of quotients of $G$ isomorphic to $\mathbb{Z}/p\mathbb{Z}$ is $$(p^{d_p(G)}-1)/(p-1) \geq p^{d_p(G)-1},$$ and the restriction to $G$ of the action of each of these copies of $\mathbb{Z}/p\mathbb{Z}$ on $M_p$ is different because the kernel of the action is different.
	
	By Linnick's theorem \cite{Linnick}, there is a constant $c$ such that, for all primes $p$, there is a prime $q \equiv 1$ mod $p$ with $q \leq p^c$. Standard arguments of representation theory show that $\mathbb{Z}/p\mathbb{Z}$ has non-trivial modules in dimension $1$ over $\mathbb{F}_q$, giving at least $p^{d_p(G)-1}$ non-isomorphic irreducible $G$-modules of order $\leq p^c$. But $p^{d_p(G)-1}$ grows faster than polynomially in $p^c$ for any $c$, because $d_p(G)$ is unbounded.
\end{proof}

\begin{prop}
	Let $G$ be a finitely generated profinite group without UBERG. Then there is a finitely generated profinite group $\widetilde{G}$ which is abelian-by-$G$ that is not $PFP_1$.
\end{prop}
\begin{proof}
	Let $A$ be  the product over all finite simple $G$-modules. Note that $A$ a $1$-generated profinite $G$-module.
	Let $\widetilde{G} = A \rtimes G$. Then $A$ is a finitely normally generated subgroup and with finitely generated factor $G$, so that $\widetilde{G}$ is finitely generated.
	
	We claim that $\widetilde{G}$ is not $PFP_1$.
	Let $M$ be a simple $G$-module, then (as above) we have the exact sequence
	\[
	0 \longrightarrow H^1(G,M) \longrightarrow H^1(\widetilde{G},M) \longrightarrow \Hom_G(A,M) \longrightarrow 0.
	\]
	Since $M$ is a factor of $A$ we have $|H^1(\widetilde{G},M)| \geq |\End_G(M)|$.
	Summing over all modules of order $p^c$ we obtain
	\[
	\sum_{|M|=p^c} |H^1(\widetilde{G},M)|-1 \geq (p-1) r_c(G,\bbF_p).
	\]
	Since the last term is not polynomially bounded in $p^c$, we deduce that $\widetilde{G}$ is not of type $PFP_1$ (see \cite{CCV}).
\end{proof}

\section{Universal Frattini covers and examples}\label{sec:examples}

\begin{thm}
	Suppose $\pi: H \to G$ is a Frattini cover of $G$ (that is, an epimorphim with $\ker\pi \leq \Phi(H)$). Then $H$ has type $\FP_1$ if and only if $G$ does.
\end{thm}
\begin{proof}
	Clearly if $H$ has type $\FP_1$, $G$ does. For the converse, we use Proposition~\ref{prop:augmentation}.
	
	Any non-Frattini chief factor of $H$ is a non-Frattini chief factor of $G$. So if $M$ is an irreducible $H$-module, either $\delta_H(M) = 0$ or $M$ is a $G$-module, in which case there is some $d$ such that $\delta_H(M)/r_H(M) = \delta_G(M)/r_G(M) \leq d$.
\end{proof}

\begin{cor}
	Suppose $G$ has type $\FP_1$. The universal Frattini cover $\tilde{G}$ of $G$ has type $\FP$.
\end{cor}
\begin{proof}
	$\tilde{G}$ has type $\FP_1$, and since it is projective, it has cohomological dimension $1$.
\end{proof}

This is interesting because we have profinite groups of type $\FP_1$ which are not finitely generated. Proposition \ref{prop:augmentation} shows that any infinite product of non-abelian finite simple groups has type $\FP_1$; \cite[Example 2.6]{Damian} gives $A = \prod A_5$, a countably infinite product of copies of $A_5$, as an example which is not finitely generated -- but countability is not needed: we can take products over indexing sets of arbitrary cardinality and the result still holds. So the universal Frattini cover $\tilde{A}$ of $A$ has type $\FP$ but is not finitely generated.

Analogously to these results for type $\FP_1$, we may prove results for type $\PFP_1$.

\begin{thm}
	Suppose $\pi: H \to G$ is a Frattini cover of $G$ (that is, an epimorphim with $\ker\pi \leq \Phi(H)$). Then $H$ has type $\PFP_1$ if and only if $G$ does.
\end{thm}
\begin{proof}
	Clearly if $H$ has type $\PFP_1$, $G$ does. For the converse, we use \cite[Lemma 5.2]{GKKL}.
	
	Recall that the kernel of a Frattini cover is pronilpotent by \cite[Corollary 2.8.4]{RZ}. So for any group $L$ associated with a characteristically simple non-abelian group $A$, $\rk_L(H) = \rk_L(G)$, and for any $H$-module $M$, the non-Frattini chief factors of $H$ $H$-isomorphic to $M$ are precisely the non-Frattini chief factors of $G$ $G$-isomorphic to $M$. By Theorem \ref{thm:PFP1}, if $G$ has type $\PFP_1$, $H$ does too.
\end{proof}

\begin{cor}
	Suppose $G$ has type $\PFP_1$. The universal Frattini cover $\tilde{G}$ of $G$ has type $\PFP$.
\end{cor}
\begin{proof}
	$\tilde{G}$ has type $\PFP_1$, and since it is projective, it has cohomological dimension $1$.
\end{proof}

We finish by using these results to construct groups $G$ of type $\PFP_1$ that do not have UBERG: then the universal Frattini cover $\tilde{G}$ of $G$ has type $\PFP$ but does not have UBERG.
We first give an example with $G$ not finitely generated.

\subsection{Products of special linear groups}
For every prime number $p$ let $m(p)$ be a non-negative integer. We consider the profinite group
\[
	G = \prod_{p} \SL_2(\bbF_p)^{m(p)}.
\]

\begin{thm}\label{thm:product-of-sl2s}
The group $G = \prod_{p} \SL_2(\bbF_p)^{m(p)}$ has the following properties:
\begin{enumerate}[(i)]
\item\label{it:sl2s-finitely-generated} $G$ is finitely generated if and only if $m(p)$ grows at most polynomially in~$p$.
\item\label{it:sl2s-pfp1} $G$ is $\PFP_1$ if and only if $m(p)$ grows at most exponentially in $p$.
\item\label{it:sl2s-pfp2} If $G$ is $\PFP_2$ then $G$ is finitely generated.  
\item\label{it:sl2s-uberg} $G$ has UBERG if and only if $G$ is finitely generated.
\end{enumerate}
\end{thm}
\begin{cor}
For $m(p) = 2^p$ the group $G$ is $\PFP_1$ but not finitely generated and not $\PFP_2$. 
The universal Frattini cover of $G$ is an infinitely generated projective profinite group of type $\PFP$.
\end{cor}

We collect some observations in order to prove the theorem.
\begin{prop}\label{prop:sl2-no-small-reps}
Let $p\geq 5$ be a prime number. Let $k$ be a field and let $V$ be a non-trivial simple $k[\SL_2(\bbF_p)]$-module.
\begin{enumerate}[(i)]
\item If $\mathrm{char}(k)\neq p$, then $\dim_k(V) \geq \frac{1}{2}(p-1)$.
\item If $\mathrm{char}(k)= p$ and $H^1(\SL_2(\bbF_p),V) \neq 0$, then $\dim_k(V) \geq p-2$.
 \end{enumerate}
\end{prop}
\begin{proof}
Assume that $\mathrm{char}(k)\neq p$. Extending scalars, we may assume that $k$ is algebraically closed. Every absolutely irreducible representation of $\SL_2(\bbF_p)$ gives rise to an irreducible projective representation of $\mathrm{PSL}_2(\bbF_p)$ and so the assertion follows from \cite[p.234]{Seitz-Zalesskii}.

Assume that $\mathrm{char}(k) = p$.
Let $M_\lambda$ be the irreducible representation of weight $\lambda \in \{0,\dots, p-1\}$. Recall that $M_\lambda$ has degree $\lambda+1$; in particular $M_0$ is the trivial representation.
Let $R_\lambda$ denote the associated PIM. It is known that the composition factors of $R_\lambda$ are 
$M_\lambda, M_{p-1-\lambda}, M_{\lambda}, M_{p-3-\lambda}$ (the last term only occurs if $p-3-\lambda > 0$); see \cite{Humphreys} and references therein.
Let 
\[
	\cdots \longrightarrow P_2 \longrightarrow  P_1 \longrightarrow  P_0 \longrightarrow  k
\]
be the minimal projective resolution of $k$ as $k[\SL_2(\bbF_p)]$-module. Then 
\[
	H^1(\SL_2(\bbF_p),V) = \Hom_{k[\SL_2(\bbF_p)]}(P_1,V) 
\]
In this case $P_0 = R_0$ and $P_1$ is the projective cover of the kernel $W$ of $R_0 \to k$. As such every homomorphism into a simple module factors through $W$ and we have
\[
	H^1(\SL_2(\bbF_p),V) = \Hom_{k[\SL_2(\bbF_p)]}(W,V). 
\]
Every simple factor of $W$ is a composition factor of $R_0$. If $H^1(\SL_2(\bbF_p),V) \neq 0$, then $V$ is either $M_{p-3}, M_{p-1}$ or the trivial module $M_0=k$. However, for $p\geq 5$ the group $\SL_2(\bbF_p)$ is perfect and so $H_1(\SL_2(\bbF_p),M_0) = 0$. We conclude that $V$ is $M_{p-1}$ or $M_{p-3}$ and has dimension at least $\dim_k M_{p-3} = p-2$.
\end{proof}
\begin{lem}\label{lem:existence-small-module-sl2}
Let $p\geq 3$ be an odd prime. There is an irreducible representation $V$ of $\SL_2(\bbF_p)$ over $\bbF_2$ such that 
\[
	\dim_{\bbF_2} (V) \leq p  \quad \text{ and } \quad H^1(\SL_2(\bbF_p), V) \neq 0.
\]
Moreover, the centre of $\SL_2(\bbF_p)$ acts trivially on $V$.
\end{lem}
\begin{proof}
Let $B \leq \SL_2(\bbF_p)$ the subgroup of upper triangular matrices.
Since $\bbF_p^\times$ is a factor of $B$, we have $H^1(B,\bbF_2) = \Hom(B,\bbF_2) \neq 0$. Let $M =\mathrm{Ind}_B^{\SL_2(\bbF_p)}(\bbF_2)$ be the induced representation. The centre of $\SL_2(\bbF_p)$ is contained in $B$ and acts trivially on $M$.
By Shapiro's Lemma
\[
  H^1(\SL_2(\bbF_p), M) \cong H^1(B,\bbF_2) \neq 0.
\]
Let $V_0$ be a simple factor of $M$ and consider the exact sequence
\[
	0 \longrightarrow M' \longrightarrow M \longrightarrow V_0 \longrightarrow 0.
\]
The associated long exact sequence shows that at least one of $H^1(\SL_2(\bbF_p),V_0)$ and $H^1(\SL_2(\bbF_p),M')$ is non-trivial.
By induction, we deduce that some composition factor $V$ of $M$ also satisfies $H^1(\SL_2(\bbF_p),V)\neq 0$. Since $\dim_{\bbF_2} M = p+1$ and $M$ has a trivial composition factor, it follows that $\dim_{\bbF_2} V \leq p$.
\end{proof}
\begin{cor}\label{cor:H1-length}
For every odd prime $p$
	\[
		2^{\frac{1}{2}(p-1)}\leq l^{H^1}(\PSL_2(\bbF_p)) \leq 2^p.
	\]
\end{cor}
\begin{proof}
Let $M$ be an simple module for $\PSL_2(\bbF_p)$ with $H^1(\PSL_2(\bbF_p),M) \neq 0$.
Then $H^1(\SL_2(\bbF_p),M)\neq 0$ (see Proposition~\ref{AG}) and hence 
the lower bound can be deduced from Proposition~\ref{prop:sl2-no-small-reps} and the inequality $\frac{1}{2}(p-1) \leq p-2$.
Conversely, let $V$ be the irreducible representation provided by Lemma \ref{lem:existence-small-module-sl2}. Then $V$ is non-trivial, factors through $\PSL_2(\bbF_p)$ and $H^1(\PSL_2(\bbF_p),V)\neq 0$ by Proposition~\ref{AG}. This proves the upper bound. 
\end{proof}
\begin{proof}[Proof of Theorem \ref{thm:product-of-sl2s}]
For all assertions, we may assume that $m(2)=0, m(3) = 0$; in particular all factors are quasisimple.

\medskip

\eqref{it:sl2s-finitely-generated}: 
Since $\SL_2(\bbF_p)$ for $p\geq 5$ is quasisimple and the simple factors of these groups are pairwise non-isomorphic, 
we have $d(G)= \max_{p\geq 5} d(\SL_2(\bbF_p)^{m(p)})$. It follows from
\cite[Proposition 3.4]{Thevenaz}
\[
	\left\vert d(\SL_2(\bbF_p)^{m(p)}) - \frac{\log(m(p))}{\log(|\PSL_2(\bbF_p)|)} \right\vert \leq C_p
\]
where $C_p \leq \frac{\log(2|\Aut(\PSL_2(\bbF_p))|)}{|\PSL_2(\bbF_p)|} + 1$. Using $\Aut(\PSL_2(\bbF_p)) = \mathrm{PGL}_2(\bbF_p)$ we can bound the constants $C_p$ uniformly by $C_p \leq 3$.

In particular, $G$ is finitely generated if and only if 
\[
	\frac{\log(m(p))}{\log(|\PSL_2(\bbF_p)|)} \leq c
\]
for some constant $c$ independent of $p$. Since $\log(|\PSL_2(\bbF_p)|) \sim 3 \log(p)$ as $p$ tends to infinity,  we deduce that $G$ is finitely generated exactly when $m(p)$ grows at most polynomially in $p$.

\medskip

\eqref{it:sl2s-pfp1}:
This follows from Corollary \ref{cor:H1-length} and Corollary \ref{cor:PFP1-products}.

\medskip

\eqref{it:sl2s-pfp2}: Assume that $G$ is $\PFP_2$. Let $p \geq 5$ be a prime and let $V_{\mathrm{Ad}}$ denote the adjoint representation of $\SL_2(\bbF_p)$.  We have $\dim_{\bbF_p} V_{\mathrm{Ad}} = 3$ and 
\[
	H^2(\SL_2(\bbF_p), V_{\mathrm{Ad}}) \neq 0,
\]
since $\SL_2(\mathbb{Z}/p^2\mathbb{Z})$ is a non-split extension of $V_{\mathrm{Ad}}$ by $\SL_2(\bbF_p)$.
Inflating these representations to $G$ we obtain
\[
	\sum_{|W| = p^3} |H^2(G,W)| -1 \geq m(p) 
\]
Since $G$ is assumed to have $\PFP_2$, there is a constant $c$ independent of $p$ such that
\[
	p^{3c} \geq m(p)
\]
and $m(p)$ grows at most polynomially in $p$. By \eqref{it:sl2s-finitely-generated} the group $G$ is finitely generated.

\eqref{it:sl2s-uberg}: Note that $G$ does not involve every finite group as a continuous subfactor. Therefore, assuming that $G$ is finitely generated, it follows from \cite[Theorem 6.10]{KV} that $G$ has UBERG. 
Conversely, assume that $G$ has UBERG. Every $\SL_2(\bbF_p)$ has a an irreducible representation of dimension $2$ over $\bbF_p$. In particular, $G$ has at least $m(p)$ such representations and we conclude that $m(p) \leq p^{2c}$ for a constant which does not depend on $p$.
\end{proof}

\subsection{Products of alternating groups}
In this section we study $\PFP_1$ for products of alternating groups.

\begin{thm}
	\label{altlength}
	For large primes $b$, $l^{H^1}(Alt(b)) = b^{b-2}$
\end{thm}
\begin{proof}
	Our approach is to first enumerate the $Alt(b)$-modules smaller than $b^{b-2}$, and then show they have trivial first cohomology.
	
	By \cite{James} (restricting representations from symmetric to alternating groups), in characteristic $p$, for large $b$, the only non-trivial representation of $Alt(b)$ smaller than $p^{b^2/4}$ is the fully deleted permutation module described in \cite[Section 5.3, Alternating groups]{KL}. For large $b$, this is greater than $b^{b-2}$, so we only need to consider the fully deleted permutation modules $M_p$ over $\F_p$ for primes $p < b$: for $p \geq b$, the fully deleted permutation module has size $\geq b^{b-2}$.
	
	We refer the reader to \cite[Section 4.6]{Martin} for background on Young modules, and \cite[Section 5.1]{Martin} for which ones belong to the principal block: given a field $\F_p$ and $\lambda \vdash b$, the Young module $Y^\lambda$ (over $p$) is the indecomposable summand of the permutation module $M^\lambda$ containing the Specht module $S^\lambda$, which, when $\lambda$ is a restricted partition of $b$, has as its unique simple quotient the simple module $D^\lambda$.
	
	Fix a field $\F_p$, for $p < b$ prime. Using \cite{James}, we can describe the $\F_p[Sym(b)]$-module $\Ind^{Sym(b)}_{Alt(b)} M_p$. Since $|Sym(b):Alt(b)|=2$, using Clifford theory, there are at most $2$ irreducible $\F_p[Sym(b)]$-modules restricting to the $\F_p[Alt(b)]$-module $M_p$. In odd characteristic, these are $D^{(b-1,1)}$ and $D^{(b-1,1)} \otimes \mathrm{sgn}$, where $\mathrm{sgn}$ is the sign representation. (Recall that $M^{(b-1,1)}$ is the natural $b$-dimensional permutation module for $Sym(b)$.) In characteristic $2$, $\mathrm{sgn}$ is trivial, and it is easy to see by counting dimensions that $D^{(b-1,1)}$ is the only irreducible $\F_p[Sym(b)]$-module restricting to $M_p$. In either case, by the standard argument, these are the only possible composition factors of $\Ind^{Sym(b)}_{Alt(b)} M_p$.
	
	Since $p \nmid b$, the argument of \cite[Section 5.3, Alternating groups]{KL} shows that $D^{(b-1,1)} = S^{(b-1,1)}$ is a direct summand of $M^{(b-1,1)}$, and hence it is the Young module $Y^{(b-1,1)}$.
	
	For $b$ odd, $Y^{(b-1,1)}$ is not in the principal block for $p=2$, so $H^1(Sym(b),Y^{(b-1,1)})=0$. For $3 \leq p$, \cite[6.3, Corollary]{KN} shows $H^1(Sym(b),Y^\lambda)=0$ for all $\lambda$. Finally, \cite[Theorem 2.4]{BKM} shows for $3 \leq p$ that $H^1(Sym(b),S^{(b-1,1)} \otimes \mathrm{sgn})=0$. We conclude that,for all primes $p<b$, $H^1(Sym(b),D^{(b-1,1)})=H^1(Sym(b),D^{(b-1,1)} \otimes \mathrm{sgn})=0$, so $H^1(Sym(b),\Ind^{Sym(b)}_{Alt(b)} M_p)=0$, and hence by Shapiro's lemma $H^1(Alt(b),M_p)=0$. (\cite[6.3, Corollary]{KN} states that  $H^i(Sym(b),Y^\lambda)=0$ for all $\lambda$ and all $1 \leq i \leq 2p-3$, but in fact the correct bound is $1 \leq i \leq 2p-4$: see \cite[2.4]{HN}.)
	
	The smallest $Alt(b)$-module not yet accounted for is the fully deleted permutation module $M_b$, and $|M_b| = b^{b-2}$. Therefore $l^{H^1}(Alt(b)) \geq b^{b-2}$. Moreover, by \cite[5.1, Lemma]{KN}, $H^1(Sym(b),D^{(b-1,1)}) \neq 0$, and we deduce as above that $H^1(Alt(b),M_b) \neq 0$, so $l^{H^1}(Alt(b))=b^{b-2}$.
\end{proof}

\begin{thm}
	\label{thm:prod-alt}
 Let $f \colon \mathbb{N}_{\geq 5} \to \mathbb{N}$ and let
\[
	G = \prod_{b \geq 5} Alt(b)^{f(b)}.
\]
Then the following are equivalent:
\begin{enumerate}[(i)]
\item $G$ is finitely generated,
\item $G$ is of type $\PFP_1$,
\item there is $c> 0$ such that $f(b) \leq (b!)^{c}$ for all $b$.
\end{enumerate}
\end{thm}
\begin{proof}
Indeed, $G$ is finitely generated if and only if $f(b) \leq (b!)^c$ for some $c$, by \cite[Section 1]{Hall} (since $|\mathrm{Out}(Alt(b))| \leq 4$ for all $b$).
By Corollary \ref{cor:PFP1-products} and Theorem \ref{altlength}, $G$ has type $\PFP_1$ if and only if $f(b) \leq b^{(b-2)c'}$ for some $c'$. Since $(b!)^c \leq b^{(b-2)c} \leq (b!)^{2c}$ by Stirling's approximation, this is equivalent to statement~(iii).
\end{proof}
Using this we obtain a finitely generated example which has type $\PFP_1$ but not UBERG. This example has superexponential subgroup growth, in contrast to the PFG and UBERG conditions which both, for finitely generated groups, imply at most exponential subgroup growth (see \cite[Theorem 10.2]{JP}, \cite[Proposition 5.4]{KV}).
\begin{cor}
	\label{superexponential}
	Let $G = \prod_{b \geq N, b \text{ prime}} Alt(b)^{b!/8}$, for some large $N$. Then $G$ is $2$-generated, has superexponential subgroup growth, does not have UBERG and has type $\PFP_1$.
\end{cor}
\begin{proof}
By Theorem~\ref{thm:prod-alt} $G$ is finitely generated and has type $\PFP_1$.
The group $G$ is $2$-generated by \cite[Corollary 1.2]{MRD}, and has superexponential subgroup growth by \cite[Proposition 10.2]{JP}. This implies that $G$ does not have UBERG by \cite[Corollary 5.5]{KV}. 
\end{proof}
\begin{rem}
	Once again, the universal Frattini cover of $G$ in Corollary~\ref{superexponential} is a finitely generated projective profinite group of type $\PFP$ with superexponential subgroup growth, which cannot occur for PFG or UBERG groups.
\end{rem}

\section{\texorpdfstring{$\FP_1$}{FP1} and UBERG}\label{sec:UBERG_FP1}

Given all the examples above, one remaining gap is an example of an infinitely generated group which has UBERG and type $\PFP_1$. It is surprising to discover that no such examples exist.

\begin{lem}
	\label{linorder}
	There is some constant $f$ such that, for any non-abelian simple group $S$, $l^{lin}(S) \leq |S|^f$.
\end{lem}

Note that there is no constant $f'$ such that $|S|^{f'} \leq l^{lin}(S)$ for all $S$: the alternating groups give a counter-example.

\begin{proof}
	This is trivial for the sporadic groups. For the alternating groups, it follows from $l^{lin}(Alt(n)) \leq 2^{n-1} < n!/2$ for $n \geq 5$. So we may suppose $S$ is a group of Lie type.
	
	Suppose $S$ is defined over $\F_q$, $q$ a power of a prime $p$. We conclude from \cite[Proposition 5.4.6, Remark 5.4.7]{KL} that $\F_q$ is the smallest field over which a non-trivial irreducible representation of $S$ of minimal dimension $k$ in characteristic $p$ is realised, and any irreducible representation of $S$ in characteristic $p$ has size at least $q^k$. By comparing \cite[Proposition 5.4.13]{KL} and \cite[Theorem 5.3.9]{KL}, we see that except possibly in finitely many cases (which we can ignore), the smallest non-trivial irreducible representation of $S$ occurs in characteristic $p$. So $l^{lin}(S) \leq q^{k^2}$, where the value of $k$ for each $S$ can be read from \cite[Table 5.4.C]{KL}. Meanwhile, $|S|$ can be read from \cite[Table 5.1.A, Table 5.1.B]{KL}: comparing these gives the result.
\end{proof}

\begin{thm}\label{thm:UBERG+FP_1=FG}
	Suppose $G$ is a profinite group with UBERG and type $\FP_1$. Then $G$ is finitely generated.
\end{thm}
\begin{proof}
	We may assume $d(G) > 2$; otherwise the result is trivial.
	
	From \cite[Theorem 1.4]{VL}, it follows that for any finite group $H$ with $d(H) > 2$, there is a monolithic primitive group $L$ with minimal normal subgroup $K$ such that the crown-based power $L_k$ is isomorphic to a quotient of $H$, $d(H) = d(L_k)$ for some $k$, and for any proper quotient $J$ of $L_k$, $d(J) < d(L_k)$. Write $G$ as an inverse limit $\varprojlim_{i \in \mathbb{N}} H_i$ of finite groups (this is possible because $G$ has UBERG, so it is countably based by \cite[Proposition 1.3]{CCV}). So $d(G) = \sup_i d(H_i) = \sup_i d((L_i)_{k_i})$, where $L_i, k_i$ are chosen for each $H_i$ with $d(H_i) > 2$ as above.
	
	For each $i$, let $K_i$ be the unique minimal normal subgroup of $L_i$. Suppose first that $K_i$ is abelian. Since $L_i$ is a quotient of $G$, and $G$ has type $\FP_1$, by Corollary \ref{crownFP1} there is some constant $d$ such that $k_i \leq dr_G(K_i)$. By \cite[Proposition 6]{VLM}, and in the notation used there, since $d(L_i/K_i) < d((L_i)_{k_i})$, $d((L_i)_{k_i}) = h_{L_i,k_i}$. But $h_{L_i,k_i} \leq d+2$, by \cite[Theorem A]{AG}.
	
	Now suppose that $K_i$ is non-abelian, $K_i = S^r$ with $S$ simple. Because $G$ has UBERG, there is some $c$ such that $k_i \leq l^{lin}(S)^{cr}$ for all $L$. By \cite[Corollary 8]{VLM}, for $s \geq \max(2,d(L_i/K_i))$, $d((L_i)_{k_i}) \leq s$ if and only if $k_i \leq \psi_{L_i}(s)$ (where $\psi_{L_i}(s)$ is defined in \cite{VLM}). For $\gamma$ the constant defined in \cite[Proposition 9]{VLM}, we have $k_i \leq l^{lin}(S)^{cr} \leq |S|^{cfr}$ by Lemma \ref{linorder}, which is $\leq \gamma|S|^{r(s-2)}$ when $$s \geq cf - \log(\gamma)/\log(|K_i|)+2,$$ and $$\gamma|S|^{r(s-2)} \leq (\gamma|S^r|^{s-1})/(r|\mathrm{Out}(S)|) \leq \psi_{L_i}(s)$$ by \cite[Lemma 7.7]{GMP} and \cite[Proposition 10]{VLM}. Overall, we get that $$d((L_i)_{k_i}) \leq \max\{d(L_i/K_i), cf - \log(\gamma)/\log(|K_i|)+2\}.$$ But $d(L_i/K_i) < d((L_i)_{k_i})$, so $$d(L_i^{(k_i)}) \leq cf - \log(\gamma)/\log(|K_i|)+2 \leq cf - \log(\gamma)+2,$$ and $d(G) \leq \max\{d+2, cf - \log(\gamma)+2\}$.
\end{proof}

We know of no direct proof of this fact, without the use of crown-based power characterisations. Recalling that UBERG plus type $\FP_1$ is equivalent to APFG, we conclude that APFG implies finite generation.

\end{document}